\documentclass{amsproc}
\usepackage{amsmath}
\usepackage{amssymb}
\usepackage{amsthm}
\usepackage{latexsym}
\usepackage{graphicx}
\usepackage{color}
\newenvironment{adding}[1]{\marginpar{{\bf [}\footnotesize #1}\bf }{\marginpar{\bf ]}}
\newenvironment{addingwithoutmargin}{\bf }{}

\newcommand{\badd}[1]{\begin{adding}{#1}}
\newcommand{\eadd}{\end{adding}}
\newcommand{\baddw}{\begin{addingwithoutmargin}}
\newcommand{\eaddw}{\end{addingwithoutmargin}}

\catcode`\@=11\def\th@mydefinition{%
  \thm@notefont{\bfseries}
}
\catcode`\@=12
\newtheorem{theorem}{Theorem}[section]
\newtheorem{prop}[theorem]{Proposition}
\newtheorem{lemma}[theorem]{Lemma}
\newtheorem{pty}[theorem]{Property}
\newtheorem{corollary}[theorem]{Corollary}

\newtheorem{problem}[theorem]{Problem}
\theoremstyle{mydefinition}
\newtheorem{example}[theorem]{Example}
\theoremstyle{definition}
\newtheorem{ex}[theorem]{Example}
\newtheorem{definition}[theorem]{Definition}
\newtheorem{re}[theorem]{Remark}

\newcommand{\wdim}{\operatorname{dim}_{\mathrm w}}
\newcommand{\edim}{\operatorname{ed}}
\newcommand{\edimw}{\edim_{\mathrm{w}}}
\newcommand{\edimgm}{\edim_{\mathrm{GM}}}
\newcommand{\edimtrop}{\edim_{\mathrm{t}}}
\newcommand{\set}[2]{\{#1;\,#2\}}
\newcommand{\sgn}{\operatorname{sgn}}
\newcommand{\NEW}[1]{{\em #1\/}\index{#1}}
\newcommand{\adj}{^{\mathrm{adj}}}
\newcommand{\adjp}{^{\mathrm{adj}+}}
\newcommand{\adjm}{^{\mathrm{adj}-}}
\newcommand{\balance}{\,\nabla\, }
\newcommand{\bal}{\balance}
\newcommand{\notbalance}{\,\nabla\!\!\!\!/\ }
\newcommand{\rkdet}{\rk_{\mathrm{det}}}
\newcommand{\dett}[1]{\operatorname{bidet}(#1)}
\newcommand{\detp}[1]{|#1|}
\newcommand{\mymod}[1]{\operatorname{m}(#1)}
\newcommand{\card}[1]{\# #1}
\def\allperm{\mathfrak{S}}
\def\permeven{\mathfrak{A}}
\def\mr{\operatorname{mr}}
\def\mc{\operatorname{mc}}
\def\er{\operatorname{er}}
\def\ec{\operatorname{ec}}
\newcommand{\mrgm}{\mr_{\mathrm{GM}}}
\newcommand{\mcgm}{\mc_{\mathrm{GM}}}
\newcommand{\mrtrop}{\mr_{\mathrm{t}}}
\newcommand{\mctrop}{\mc_{\mathrm{t}}}
\newcommand{\mrw}{\mr_{\mathrm{w}}}
\newcommand{\mcw}{\mc_{\mathrm{w}}}
\newcommand{\erw}{\er_{\mathrm{w}}}
\newcommand{\ecw}{\ec_{\mathrm{w}}}
\newcommand{\ergm}{\er_{\mathrm{GM}}}

\newcommand{\ertrop}{\er_{\mathrm{t}}}

\def\sr{\operatorname{sr}}
\def\ror{\operatorname{r}}
\def\cor{\operatorname{c}}
\def\cosp{\operatorname{\mathcal C}}
\def\rosp{\operatorname{\mathcal R}}
\newcommand{\trop}{\operatorname{trop}}
\def\termrk{\operatorname{term}}
\def\facrk{\operatorname{f}}

\def\d{\begin{proof}}
\def\dd{\end{proof}}
\newcommand{\rk}{\operatorname{rk}}
\newcommand{\mymathbf}{}
\def\cal#1{\mathcal #1}
\makeatletter
\@addtoreset{equation}{section}

\makeatother
\newcommand{\rmax}{\mathbb{R}_{\max}}

\newcommand{\smax}{\mathbb{S}_{\max}}

\def\SS{{\cal S}}
\def\MM{{\cal M}}
\def\Mnm{{\cal M}_{mn}}
\def\e{\begin{equation}}
\def\ee{\end{equation}}
\def\RR{{\cal R}}
\def\M{{\cal M}_{n}}
\newcommand{\Z}{{\mathbb Z}}
\newcommand{\B}{{\mathbb B}}
\def\R{{\mathbb R}_{\max}}
\def\Ub{\Re^{\nu}}

\def\Ti{{\mathbb T}_{\mathrm e}}
\def\Re{{\mathbb R}}
\def\S{{\mathbb S}_{\max}}
\def\BB{{\mathbb B}^{{\mathrm s}}}

\newcommand{\N}{\mathbb{N}}
\newcommand{\Q}{\mathbb{Q}}

\newcommand{\per}{\operatorname{per}}

\DeclareMathAlphabet{\mathbbold}{U}{bbold}{m}{n}
\newcommand{\zero}{\mathbbold{0}}
\newcommand{\unit}{\mathbbold{1}}
\newcommand{\oo}{0}
\newcommand{\oi}{1}
\newcommand{\ooo}{\zero}
\newcommand{\ooi}{\unit}
\begin{document}

\title{Linear independence over tropical semirings and beyond}
\author{Marianne Akian}
\address{Marianne Akian,
INRIA Saclay--\^Ile-de-France and CMAP, \'Ecole 
Polytechnique. Address:
CMAP, \'Ecole Polytechnique,
Route de Saclay,
91128 Palaiseau Cedex, France.}
\email{Marianne.Akian@inria.fr}
\author{St\'ephane Gaubert}
\address{St\'ephane Gaubert,
INRIA Saclay--\^Ile-de-France and CMAP, \'Ecole 
Polytechnique. Address: CMAP, \'Ecole Polytechnique,
Route de Saclay,
91128 Palaiseau Cedex, France.}
\email{Stephane.Gaubert@inria.fr}
\author{Alexander Guterman}
\address{Alexander Guterman, Moscow State University, Leninskie Gory, 119991, GSP-1,
Moscow, Russia}
\email{guterman@list.ru}
%
\thanks{The two first authors were partially supported by the joint RFBR-CNRS grant 05-01-02807.}
\thanks{The third author was partially supported by the invited professors program from INRIA Paris-Rocquencourt and by the grants MK-2718.2007.1 and RFBR 08-01-00693a.}

\date{\today}
\begin{abstract}
We investigate different notions of linear independence
and of matrix rank that are relevant for max-plus or tropical semirings.
The factor rank and tropical rank have already received attention,
we compare them with the ranks defined in terms of signed tropical
determinants or arising from a notion of linear independence introduced by Gondran and Minoux.
To do this, we revisit the symmetrization of the max-plus algebra,
establishing properties of linear spaces, linear systems,
and matrices over the symmetrized max-plus algebra. In parallel we develop some general technique to prove combinatorial and polynomial identities for matrices over semirings that we illustrate by a number of examples.
\end{abstract}
\maketitle


\section{Introduction}
The max-plus semiring $\rmax$ is the set of real numbers, 
completed by $-\infty$, equipped with the
addition $(a,b)\mapsto \max(a,b)$ and the multiplication $(a,b)\mapsto a+b$.
The name ``tropical'' is now used either as a synonym of ``max-plus'',
or in a wider sense, referring to algebraic
structures of a similar nature.

Max-plus structures have appeared in relation with
in a variety of fields, like operations research and optimization~\cite{vorobyev67,CG,gondranminouxbook}, discrete
event systems~\cite{BCOQ92,ccggq99}, automata theory (see~\cite{pin} and its references), quasi-classic asymptotics~\cite{maslov73}, Hamilton-Jacobi partial differential equations and optimal control \cite{duality,maslovkolokoltsov95,litvinov00,mceneaney}, and more recently, tropical algebraic
geometry, see in particular~\cite{viro,mikhalkin,kapranov,passare,itenberg,RGST}.


This has motivated the investigation of
the analogues of basic questions in algebra, among
which linear independence is an elementary but central one.
The study
of max-plus linear independence goes back to the work
of Cuninghame-Green~\cite{CG}, who defined a family
to be {\em weakly independent} if no vector of the family is a linear
combination of the others. This notion was further
studied by Moller~\cite{Mol} and Wagneur~\cite{Wag}, who
showed that a finitely generated linear subspace
of $\rmax^n$ admits a weakly independent generating family
which is unique up to a reordering and a scaling of its vectors.
This result was made more precise in further works
by Butkovi\v{c} and Cuninghame-Green~\cite{CGB}, Gaubert and Katz~\cite{GK},
Butkovi\v{c}, Sergeev and Schneider~\cite{BSS}. They developed
a theory of extreme rays of max-plus linear spaces or ``cones''
(as in classical convexity, a ray is the set of scalar multiples of a single vector). This theory shows in particular that the weakly independent
generating family can be identified to the set of extreme rays.

\def\<#1>{\langle#1\rangle}
The analogy with classical convex geometry can even
be made more formal by noting that the map $\langle\cdot\rangle$ which associates to
a finite set of rays 
in $\rmax^n$ the linear span that it generates
satisfies the {\em anti-exchange axiom}:
\[
y,z\not\in \<X>,\; z\in \<X\cup\{y\}>\implies y\not\in \<X\cup\{z\}>
\]
whereas the classical linear span of a set of vectors
satisfies the {\em exchange axiom}
\[
y\not\in\<X>,y\in \<X\cup\{z\}>\implies z\in \<X\cup \{y\}> \enspace .
\]
In other words, weak independence belongs to the theory of antimatroids
rather than matroids~\cite{korte}.

Gondran and Minoux~\cite{GM2} defined a different notion of independence,
which turns out to be closer to the classical one. A finite family
is linearly dependent in their sense if it can be partitioned
in two families that generate linear spaces with an intersection
that is not reduced to the zero vector. They gave a remarkable
characterization of the families
of $n$ vectors of $\rmax^n$ which are linearly
dependent, in terms of the ``vanishing'' of their ``determinant''
(this condition requires the sum of the weights of odd and even permutations to coincide).

The theorem of Gondran and Minoux was extended 
in a paper published by an imaginary researcher
named M. Plus~\cite{Plus}, in which a symmetrized
max-plus semiring $\smax$ was introduced, as well as a generalization
of the notion of linear systems, in which the equality relation
is replaced by a ``balance'' relation denoted $\balance$. 
The symmetrized max-plus semiring comprises, besides the max-plus
numbers, some  ``negative''
and ``balanced'' numbers. A negative solution $x$ to a balance relation of the form
$a\oplus x\balance b$ means that the equation in which $x$ is put on the other
side of the equality, $a=x\oplus b$, has a solution over $\rmax$.
On this basis, M.~Plus developed an elimination technique, 
allowing him to establish generalizations
of the Gondran-Minoux theorem, as well as analogues
of several results of linear algebra including a ``Cramer rule''~\cite{Plus} (see also~\cite{maxplus90a,gaubert92a,BCOQ92,maxplus97}).  

In the recent work on tropical geometry, a different
notion of independence has emerged: a family
of vectors is said to be {\em tropically dependent}
if we can make a linear (i.e.\ max-plus) combination of its vectors
in such a way that the maximum in every row is attained
at least twice. Richter-Gebert, Sturmfels and Theobald~\cite{RGST} established
an analogue of Cramer theorem which applies to systems of ``tropical''
equations (in which the maximum in every row is required to be
attained at least twice, rather than being equal to the zero element
of the semiring). Izhakian~\cite{Izh05,Izh} introduced an extension of the max-plus semiring, 
which is somehow reminiscent of the symmetrized max-plus semiring,
and has two kind of elements, the ``real'' ones (which can
be identified to elements of the max-plus semiring) and some
``ghost'' elements which are similar to the ``balanced'' ones.
This allowed him to interpret the notion
of tropical linear dependence in terms of suitably generalized
equations over his extended semiring, and to relate tropical
linear independence with the ``non-vanishing'' of determinants
(here, a determinant vanishes if there are at least two permutations
of maximal weight).

In this paper, we give a unified treatment making explicit the analogy
between tropical linear independence and Izhakian's extension,
on the one hand, and  Gondran-Minoux dependence and M.~Plus's 
symmetrization of the max-plus semiring, on the other hand.
This unification yields as a byproduct some further results.

The latter analogy is best explained in terms of amoebas.
Let $K\{\{t\}\}$ denote the field of Puiseux series 
in a variable $t$ with coefficients  in a field $K$,
convergent in a neighborhood of the origin.
There is a canonical valuation $v$, sending a Puiseux
series to the opposite of the minimal exponent arising in its expansion.
This valuation is almost a morphism from $K\{\{t\}\}$ to $\rmax$,
since $v(a+b)\leq \max(v(a),v(b))$
with equality if the maximum is attained only once or if the monomials
of minimal degree of $a$ and $b$ do not cancel, whereas $v(ab)=v(a)+v(b)$.

Special instances of {\em non-archimedean amoebas}~\cite{kapranov,itenberg} 
are obtained by taking images of algebraic varieties
of $(K\{\{t\}\})^n$ by this valuation (acting entrywise), and
max-plus algebraic structures may be thought of as polyhedral or
combinatorial tools to study such amoebas.

In a nutshell, tropical linear independence and Izhakian's
extension arise when considering amoebas of linear spaces
when the field of coefficients $K=\mathbb{C}$ whereas Gondran-Minoux linear independence and M.~Plus's symmetrization arise when taking $K=\mathbb{R}$. Indeed, if the series $a_{ij},x_j\in \mathbb{C}\{\{t\}\}$,
for $1\leq i\leq n,1\leq j\leq p$,
satisfy
\begin{align}
\sum_{1\leq j\leq p} a_{ij}x_j =0 \enspace ,\;1\leq i\leq n
\label{e-eqlin}
\end{align}
it is clear that in every expression
\[
\max_{1\leq j\leq p} v(a_{ij})+v(x_j) \enspace ,\;1\leq i\leq n
\]
the maximum must be attained at least twice (otherwise, the expression~\eqref{e-eqlin} would be nonzero due to the presence of a leading monomial).
However, if all the series $a_{ij},x_j$ belong to $\mathbb{R}\{\{t\}\}$, 
they must keep a constant sign as $t\to 0^+$, leading to a more
precise relation. For instance, if the $a_{ij}$ all have a positive
sign as $t\to 0^+$, denoting $J^+$ (resp.\ $J^-$) the set of $j$ for which $x_j$ is positive (resp.\ negative) as $t\to 0^+$, we deduce that
\[
\max_{j\in J^+} v(a_{ij})+v(x_j) =  \max_{j\in J^-} v(a_{ij})+v(x_j)\enspace ,\;1\leq i\leq n
\]
showing that the columns of the matrix obtained by applying the valuation $v$
to every column of the matrix $(a_{ij})$ is linearly dependent in the Gondran-Minoux
sense.



The goal of this paper, which is intended as a survey, although it contains some
new results, is to draw attention to the symmetrization of the max-plus semiring, that we revisit in the light of the recent developments of tropical geometry. 
We show in particular that the proof of the Cramer theorem of~\cite{Plus},
relying on elimination, also yields, by a mere substitution
of Izhakian's semiring to the symmetrized semiring, a
(slightly extended version) of the tropical Cramer theorem of~\cite{RGST},
see Theorem~\ref{th-cramerI} and Corollary~\ref{coro-sturmf} below.

The proof of these ``Cramer theorems'' relies on a series of results,
and in particular, on the semiring analogues of classical determinantal
identities. We establish in Section~\ref{sec-combinat} a general
transfer principle, building on an idea of Reutenauer and Straubing~\cite{reutstraub}, which shows that the semiring analogue of every classical identity
can be proved automatically (there is no need to find a ``bijective proof'').
In the next two sections, we develop a theory of ``symmetrized semirings'',
which allow us to see both the symmetrization of the max-plus semiring
and its extension by Izhakian as special cases of a unique construction.

We note that the notion of symmetrized semiring, as well as a version of the transfer principle,
first appeared in~\cite{gaubert92a}, but at that time, tropical geometry was
not developed yet and the max-plus symmetrized semiring seemed to be the most (if not the only)
significant model of this structure, which was therefore not further investigated.
However, some extensions of the max-plus semiring like jets~\cite{finkelstein,Bapat} 
or Izhakian's extension~\cite{Izh,Izh05} appeared later on, providing further interesting models. The unification and comparison of these approaches is a novelty of the present
paper. 

In the last two sections, we examine
different notions of matrix rank that appear to be of interest in the max-plus semiring. Such an enterprise was carried out by Develin, Santos, and Sturmfels in~\cite{DSS} , but the ranks
relying on symmetrization or on Gondran-Minoux linear independence were
missing in~\cite{DSS}. We establish inequalities comparing them with the other ranks
which are relevant in the max-plus setting. We also
establish the analogues of several classical inequalities for matrix rank.

\section{Linear independence}\label{S1}

\begin{definition}
A \NEW{semiring} is a set $\SS$ with two binary operations, addition, denoted by $+$,
and multiplication, denoted by $\cdot$ or by concatenation, such 
that:
\begin{itemize}
\item $\SS$ is an abelian monoid under addition (with neutral element denoted by $\oo$ and called zero);
\item $\SS$ is a semigroup under multiplication (with neutral element denoted
by $\oi$ and called unit);
\item multiplication is distributive over addition on both sides;
\item $s \oo=\oo s=\oo$ for all $s\in \SS$.
\end{itemize}
\end{definition}


Briefly, a semiring differs from a ring by the fact that 
an element may not have an additive inverse.
The most common examples of semirings which are not rings are non-negative
integers $\N$, non-negative  rationals $\Q_+$ and
non-negative
reals ${\mathbb R}_+$ with the usual addition and multiplication. There are
classical examples of non-numerical semirings as well. Probably the first such
example appeared in the work of Dedekind \cite{Dedekind} in connection
with the algebra of ideals of a commutative ring (one can add and multiply
ideals but it is not possible to subtract them). 

\begin{definition}
A semiring or an abelian monoid $\SS$ is called \NEW{idempotent} if $a+a=a$  for all $a\in\SS$.
\end{definition}

\begin{definition}
A semiring $\SS$ is called \NEW{zero-sum free} or \NEW{antinegative} if $a+b=0$ implies $a=b=0$ for all $a,b\in\SS$.
\end{definition}

\begin{re} An idempotent semiring is necessary zero-sum free.
\end{re}

\begin{definition}
A semiring $\SS$ is called \NEW{commutative}  if the multiplication is commutative, i.e. $a\cdot b=b\cdot a$ for all $a,b\in\SS$.
\end{definition}
We shall always assume that the semiring $\SS$ is commutative.
In this paper, we mostly deal with idempotent semirings.

The most common example of idempotent semiring is the max-plus semiring
$$\R:=({\mathbb R}\cup \{ -\infty \}, \oplus, \odot),$$
where $a\oplus b= \max\{a,b\}$ and $a\odot b= a+b$.  Here the zero element of the semiring is $-\infty$, denoted by $
\ooo$, and the unit of the semiring is 0, denoted by $\ooi$. 
More generally, idempotent semirings are called
\NEW{max-plus algebras},  or \NEW{max-algebras},  or  
\NEW{tropical algebras}. They are naturally ordered by the relation $a\le  b$ if $a+b=b$. Then $a + b$ is the supremum of $a$ and $b$ for the order $\le$ and the 
neutral element for the addition 
is the minimal element for the order $\le$.
We shall be mostly interested in $\R$, but some of our considerations hold for general idempotent semirings as well.


\begin{definition} A \NEW{semimodule}, $M$, over a semiring $\SS$ 
is an abelian monoid under addition which has a neutral element, ${\mymathbf 0}$, and is equipped with a law
$$
\begin{array}{ccc} \SS \times M & \to & M \\ (s,{\mymathbf m}) & \to & s \cdot {\mymathbf m} \end{array}
$$
called \NEW{action} or \NEW{scalar multiplication} such that  for all ${\mymathbf m}$ and ${\mymathbf m}'$ in $M$ and $r,s \in \SS$
\begin{enumerate}
\item 
$(s \cdot   r)  \cdot {\mymathbf m} = s  \cdot (r  \cdot {\mymathbf m})$,
\item
$(s + r) \cdot {\mymathbf m} = s  \cdot {\mymathbf m} + r  \cdot {\mymathbf m}$, 
\item
$s\cdot  ({\mymathbf m} + {\mymathbf m}') = s\cdot  {\mymathbf m} + s\cdot  {\mymathbf m}'$,
\item
$\oi \cdot  {\mymathbf m} = {\mymathbf m}$,
\item
$s\cdot   {\mymathbf 0}={\mymathbf 0}= \oo \cdot   {\mymathbf m}$.
\end{enumerate}
\end{definition}

In the sequel, we shall often denote the action by concatenation,
omiting the symbol ``$\cdot$''.

\begin{re}
If $\SS$ is idempotent, then necessarily $M$ is idempotent.
\end{re}
\begin{re}
The usual definition of matrix operations carries over to an arbitrary semiring, which allows us to think of the set of  $m\times n$ matrices $\MM_{m,n}(\SS)$ as a semimodule over~$\SS$. When $\SS=\R$, we will denote it just by ${\cal M}_{m,n}$. Also we denote ${\cal M}_{n}(\SS)={\cal M}_{n,n}(\SS)$ 
and we identify $\SS^n$ to ${\cal M}_{n,1}(\SS)$.
\end{re}

\begin{definition}
An element ${\mymathbf m}$ in a semimodule $M$ over $\SS$ is called a \NEW{linear combination} of elements from a certain subset $P\subseteq M$ if there exists $k\ge 0$, $s_1,
\ldots, s_k\in \SS$, ${\mymathbf m}_1,\ldots, {\mymathbf m}_k\in P$ such that ${\mymathbf m}=\sum\limits_{i=1}^k s_i \cdot  {\mymathbf m}_i$ with the convention that an empty sum is equal to $\ooo$.  In this case $\sum\limits_{i=1}^k s_i \cdot  {\mymathbf m}_i$ is called a \NEW{linear combination} of the elements ${\mymathbf m}_1,\ldots, {\mymathbf m}_k$ from $P$ with coefficients $s_1,\ldots, s_k$ in $\SS$.
\end{definition}
Note that by definition all linear combinations are finite. 
\begin{definition} The \NEW{linear span}, $\langle P\rangle$, of a family or set $P$ of elements of a semimodule $M$ over a semiring $\SS$ is the set of all linear combinations of elements from $P$ with coefficients from $\SS$. We say that the family $P$ \NEW{generates} or \NEW{spans} $M$ if $\langle P\rangle=M$, and that $P$ envelopes a subset $V\subseteq M$ in $M$ if $V\subseteq \langle P\rangle$. 
\end{definition}
As over fields and rings, $\SS^n$ is spanned by the set 
$$\{ [\oi,\oo, \ldots, \oo]^t, [\oo,\oi,\oo,\ldots, \oo]^t, \ldots, [\oo,\ldots, \oo,\oi]^t \} \enspace .
$$ 
Here, and in the sequel, the transposition of vectors or matrices is denoted
by putting the symbol $t$ as a superscript.

In contrast with vector spaces over fields, there are several ways to define the notion of linear dependence over max-plus algebras. In such algebras,
a sum of non-zero vectors cannot vanish. Hence, the classical definition
cannot be used. A natural replacement is the following.

\begin{definition}[\cite{GM,GM2}] \label{D2} 
A family ${\mymathbf m}_1,\ldots, {\mymathbf m}_k$ of elements of a semimodule $M$ over a semiring $\SS$ is 
\NEW{linearly dependent  (resp.\ independent) in the Gondran-Minoux sense} if there exist (resp.\ there does not exist)
two subsets $I,J\subseteq K:=\{1,\ldots, k\}$, $I\cap J=\emptyset$, $I\cup J=K$, and scalars $\alpha_1,\ldots,\alpha_k\in\SS$, not all equal to $\oo$, such that
$\sum \limits_{i\in I} \alpha_i\cdot  {\mymathbf m}_i = \sum\limits_{j\in J} \alpha_j\cdot  {\mymathbf m}_j$.
  \end{definition}

The following notion of linear dependence can be found in~\cite{CG,Wag}, see also~\cite{CGB} and references therein.

\begin{definition} \label{D1} 
A family $P$ of elements of a semimodule $M$ over a semiring $\SS$ is 
\NEW{weakly linearly dependent (resp.\ independent)} if there is an element (resp.\ there is no element)
in $P$ that can be expressed as
a linear combination of other elements of $P$. 
\end{definition}

\begin{re} \label{RIn1} 
A family of vectors which is independent in the Gondran-Minoux sense is also independent in the weak sense. But the converse may not be true, as it is shown in the following example. 
\end{re}

\begin{ex} \label{ExD1D2}
The vectors 
  $[x_i,\ooi, -x_i]^t$ in $\R^3$, $i=1,2,\ldots,m$, are weakly linearly independent for any $m$ and for different $x_i$ (see e.g.~\cite{CGB} for details). However, by Corollary~\ref{LS} below, any four of these vectors must be linearly dependent in the Gondran-Minoux sense. 

As a concrete example, the vectors ${\mymathbf v}_i:=[i ,\ooi,-i]^t$, $i=1,2,3,4$, are linearly dependent in the Gondran-Minoux sense since
  $$ (-1) \cdot {\mymathbf v}_1 \oplus  \ooi \cdot {\mymathbf v}_3  =
\ooi \cdot  {\mymathbf v}_2 \oplus (-1) \cdot  {\mymathbf v}_4\enspace .$$
\end{ex}


\begin{definition}\label{defi-weakdim}
For a general semimodule $M$ over a general semiring $\SS$, we define the \NEW{weak dimension} of $M$ as 
$$ \wdim(M)= \min\set{\card{P}}{P\; \text{is a  weakly independent generating family of } M},$$
where $\card{P}$ denotes the cardinality of $P$ when  the set $P$ is finite and 
$\card{P}=+\infty$ otherwise.
\end{definition}
\begin{re}\label{rem-weakdim}
The weak dimension of a semimodule $M$ is equal to the minimal cardinality of a minimal generating family or the minimal cardinality of a generating family of $M$.
\end{re}

\begin{re}\label{re:weaknonincr}
Example~\ref{ExD1D2} shows that 
the weak dimension is in general not increasing.
Indeed, let $V$ be the subsemimodule of $\R^3$ 
generated by the weakly independent vectors ${\mymathbf v_i}:= [i,\ooi,-i]$ of Example~\ref{ExD1D2}. Then, $\wdim(V)=4>\wdim(\R^3)$, whereas $V\subset \R^3$.
\end{re}

Weakly independent generating families over the max-plus algebra can be obtained as follows:

\begin{definition}
An element ${\mymathbf u}$ of a semimodule $M$ over $\R$ is called an \NEW{extremal generator} (or the family $\{\lambda\cdot  {\mymathbf u}\vert \lambda\in \R\}$ is called an \NEW{extremal ray}), if the equality ${\mymathbf u}={\mymathbf v}\oplus {\mymathbf w}$ in $M$  implies that either ${\mymathbf u}={\mymathbf v}$ or ${\mymathbf u}={\mymathbf w}$.
\end{definition}
The following results shows that the subsemimodules of $\R^n$ are similar to the classical convex pointed cones.
\begin{theorem}[``Max-plus Minkowski'', {\cite[Theorem 3.1]{GK} or \cite[Proposition 24]{BSS}}] \label{T:ext_rays} Let $M$ be a closed subsemimodule of $\R^n$. Then the set of extremal generators of $M$ generates $M$, and every element of $M$ is the sum of at most $n$ extremal generators of $M$.  
\end{theorem}

\begin{corollary} \label{C1:ext_rays}
Let $M$ be a closed subsemimodule of $\R^n$. Every weakly independent generating family of $M$ is obtained by picking exactly one non-zero element in each extremal ray.
\end{corollary}
\d
By~\cite[Theorem 8]{BSS}, for a given subsemimodule $M$ of $\R^n$, any  subset $P$ of scaled vectors of $M$ is a weakly independent generating set if and only if it is equal to the set of scaled extremal generators and it generates $M$. Here a vector ${\mymathbf v}$ is called scaled if $\|(\exp v_i)_{i=1,\ldots,n}\|=1$ for some fixed norm $\|\cdot\|$. Now the result follows from Theorem~\ref{T:ext_rays}.
\dd

The same condition was obtained previously in the particular case of a finitely generated subsemimodule $M$ of $\R^n$ by Moller~\cite{Mol} and Wagneur~\cite{Wag} (see also~\cite{CGB}). In that case, Corollary~\ref{C1:ext_rays} says that the number of elements of any weakly independent generating family is the number of extremal rays. So if $M$ is a subsemimodule of $\R^n$, then $\wdim(M)$ is equal to the cardinality of any weakly independent generating family.  Example~\ref{ExD1D2} shows that this number may be arbitrary large even for vectors with 3 coordinates. Also there exists infinite weakly independent sets of such vectors.

The following observation, which was made in~\cite{gaubert98n},
emphasizes the analogy with classical convex geometry.
It shows that weak independence satisfies
the anti-exchange axiom of anti-matroids. The latter formalizes
the properties of extreme points and rays of polyhedra.
Since this axiom is valid,
Corollary~\ref{C1:ext_rays} could be recovered as a direct consequence
of the abstract Krein-Milman theorem which is established in~\cite{korte},
at least when the semimodule $M$ is finitely generated.
\begin{prop}[Anti-exchange axiom]
Let $X$ be a finite subset of $\rmax^n$, and let 
 $y, z\in \rmax^n$ be non-zero vectors such that $y,z\not\in \<X>$,
$y,z$ are not proportional in the max-plus sense,
and $y \in \<X\cup \{z\}>$. Then, $z\not\in \<X\cup \{y\}>$.
\end{prop}
\begin{proof}
Let us assume, by contradiction, that the conditions
of the proposition are satisfied and that $z\in \<X\cup \{y\}>$.
We 
can find $x\in \<X>$ 
and $\lambda\in \rmax$ such that
\[
z= x \oplus \lambda y \enspace .
\]
Since $ y \in \< X\cup  \{z\}>$, a symmetrical property
holds for $y$, namely
\[
y= x' \oplus \mu z \enspace ,
\]
for some $x'\in \<X>$ and $\mu\in \rmax$.
Substituting the latter expression of $y$ in the former equation, 
we get
\[
z=x\oplus \lambda x' \oplus \lambda\mu z \enspace .
\]
This implies that $z\geq \lambda \mu z$, here $a=(a_i)\ge b=(b_i)$ for $a,b\in  \rmax^n$ if and only if $a_i\ge b_i$ in $\Re$ for all $i=1,\ldots ,n$. 
Since the vector $z$ is non-zero, we deduce
that $\unit\geq \lambda \mu$. If the strict inequality
holds, then, we get that $z= x\oplus \lambda x'$,
showing that
$z\in \<X>$ which is a contradiction. Thus, $\lambda\mu=\unit$.
Moreover, $z\geq \lambda y\geq \lambda\mu z=z$,
and so $z=\lambda y$, which contradicts one of the assumptions.
\end{proof}

In a recent paper~\cite{Izh} Z.~Izhakian gave a new definition of linear dependence over $\R$. This definition can be extended in the following way
to the case of an arbitrary semiring.
\begin{definition} \label{D3}
A family ${\mymathbf m}_1,\ldots, {\mymathbf m}_k$, ${\mymathbf m}_i=[m_i^1,\ldots, m_i^n]^t$, $i=1,\ldots,k$, of elements of $\SS^n$ is 
\NEW{tropically linearly dependent (resp.\ independent)} if there exist (resp.\ there does not exist) 
two series of subsets $I_l,J_l\subseteq K:=\{1,\ldots, k\}$, $I_l\cap J_l=\emptyset$, $I_l\cup J_l=K$, $l=1,\ldots,n$, and scalars $\alpha_1,\ldots,\alpha_k\in\SS$, not all equal to $\oo$, such that
$\sum\limits_{i\in I_l} \alpha_i \cdot m_i^l = \sum\limits_{j\in J_l} \alpha_j\cdot m_j^l$ for all $l,\: 1\le l\le n$.
\end{definition}
\begin{re} \label{RIn2} 
A family of vectors which is tropically independent 
is also independent in the Gondran-Minoux sense.
However, the converse may not be true as it is shown in the following example.
\end{re}
\begin{ex}
  Let us consider the three vectors $[-1,\ooi,\ooi]^t$, $[\ooi,-1,\ooi]^t$, $[\ooi,\ooi,-1]^t$ in $\R^3$. These vectors are  linearly independent in the  Gondran-Minoux sense, but they are tropically linearly dependent with the coefficients $(\ooi,\ooi,\ooi)$.
\end{ex}

\begin{re} \label{Ex:no_bases}
An inconvenience  of linear independence in the Gondran-Minoux or tropical senses is that a finitely generated semimodule may not have a generating family
that is linearly independent in either of these senses.
For instance, as it will be shown later in Corollary~\ref{C:no_bases} below,
the subsemimodule $V$ of $\R^3$ 
generated by the Gondran-Minoux dependent vectors ${\mymathbf v_i}:= [i,\ooi,-i]$, already considered in Remark~\ref{re:weaknonincr}, 
contains no linearly independent generating family.
\end{re}

\begin{lemma} \label{LD1D2D3}
Let $M$ be a finitely-generated subsemimodule of $\R^n$.
\begin{enumerate}
\item If there is a generating set of $M$ which is linearly independent in the Gondran-Minoux sense  then its cardinality is the same for any such set and is equal to the cardinality of any generating set which is weakly independent.
\item  If there is a generating set of $M$ which is tropically linearly independent  then its cardinality is the same for any such set and is equal to the cardinality of any generating set which is independent in the Gondran-Minoux sense. Also in this case there is a generating set of $M$ which is weakly linearly independent and item 1 holds.
\end{enumerate}
\end{lemma}
\d
\begin{enumerate}
\item Let $B$ be a generating set of $M$  which is linearly independent in the Gondran-Minoux sense. Then by Remark~\ref{RIn1},  $B$ is weakly linearly independent. By Corollary~\ref{C1:ext_rays} any such $B$ has the same number of elements.
\item Repeats the arguments from the previous item.
\end{enumerate}
\dd

\begin{re}
Note that it is useless to consider analogues of the weak dimension for either tropical, or Gondran-Minoux linear dependence, since by Lemma~\ref{LD1D2D3} such analogues are either infinite (if there is no independent generating family) or coincide with the weak dimension.
\end{re}

\begin{corollary} \label{C:no_bases}
The semimodule $V$ from Remark~\ref{re:weaknonincr} has no generating family which is linearly independent in the Gondran-Minoux sense.
\end{corollary}
\d
From Corollary~\ref{C1:ext_rays} and Example~\ref{ExD1D2}, the cardinality of any generating family of $V$ which is weakly linearly independent  is 4. Thus by Lemma~\ref{LD1D2D3} the cardinality of a generating family which is linearly independent in the Gondran-Minoux sense, if any, should be equal to 4, which contradicts Corollary~\ref{LS} below.
\dd

In order to avoid this difficulty, one may use rather the following different notions of dimension.

\begin{definition}\label{defi-envdim}
Let $V$ be a subset of $\SS^n$, where $\SS$ is a semiring.
For every notion of independence, one can define the \NEW{enveloping dimension} of $V$ with respect to this notion as
$$ \edim(V):=\min\set{\card{P}}{P\; \text{is an independent enveloping family of } V \text{ in } \SS^n}.$$
We shall denote more specifically by
$\edimw (V)$, $\edimgm(V)$ and $\edimtrop(V)$, the enveloping dimension of $V$ with respect to the independence in the weak, Gondran-Minoux, and tropical sense, respectively.
\end{definition}

\begin{re} \label{re-env}
By definition of an enveloping family, we easily see that,
for all independence notions,
$\edim(V)=\edim(\langle V\rangle)$, $\edim(V)\leq n$ (since one can take the canonical
generating family  of $\SS^n$), and $\edim(V)\leq \edim(W)$ when $V\subset W$.
Since
\[\edimw(V)=\min\set{\wdim(M)}{V\subset M,\; M\text{ subsemimodule of } \SS^n}\enspace ,\]
we get from Remark~\ref{rem-weakdim} that
$$ \edimw(V)= \min\set{\card{P}}{P\; \text{is an enveloping family of } V \text{ in } \SS^n}.$$
Hence, when $V$ is a finite set, $\edimw(V)\leq \card{V}$.
Moreover, due to the implications between the independence notions, we have
\[  \edimw(V) \leq \edimgm (V) \leq \edimtrop(V)\enspace .\]
\end{re}

\section{Combinatorial identities in semirings}\label{sec-combinat}

\subsection{Transfer principle}
Many classical combinatorial identities which are valid for matrices over rings (like the Cayley-Ha\-mil\-ton theorem,
the Binet-Cauchy formula, or more difficult results like the
Amitsur-Levitsky identity) turn out to have semiring analogues,
which should be written ``without minus sign''. This idea
was already present in the work of Gondran and Minoux~\cite{GM2},
and it was used systematically by M.~Plus~\cite{Plus}.
Such analogues can be proved by looking for ``bijective proofs''
of these identities, along the lines of Straubing~\cite{straubing} and 
Zeilberger~\cite{zeilberger}. 
Minoux~\cite{minouxmacmahon} gave a semiring
analogue of MacMahon's master theorem. Poplin and Harwig~\cite{poplinhartwig}
gave again combinatorial proofs of several identities. 
However, an elegant observation of Reutenauer and Straubing~\cite{reutstraub}
implies that bijective proofs can be dispensed with,
since one can give a ``one line'' derivation of a valid identity in semirings from the corresponding identity in rings.
This technique, which was applied for instance in~\cite{GBCG} and in~\cite{gaubertburnside} where some semiring analogues of the Binet-Cauchy identity and of the Amitsur-Levitski identity were used, is perhaps not as well known as it should be. Hence, we state here a general transfer principle, building on the idea of Reutenauer and Straubing, and show that previously established identities
follow readily from it.


\begin{definition} A \NEW{positive polynomial expression} in the variables $x_1,\ldots,x_m$ is a formal expression produced by the context-free grammar $E\mapsto E+E,(E)\times(E),0,1,x_1,\ldots,x_m$, where the symbols $0,1,x_1,\ldots,x_m$ are thought of as terminal symbols of the grammar.
\end{definition} 
Thus, $0,1,x_1,\ldots,x_m$ are positive polynomial expressions,
and if $E$ and $F$ are positive polynomial expressions, so are
$E+F$ and $(E)\times (F)$. An example of positive polynomial expression is $E=((1+1+1)\times (1+x_1))\times (x_3)+x_2$. Every positive polynomial expression can 
be interpreted in an arbitrary semiring, by understanding the symbols $0,1,+,\times$ as the neutral elements and structure laws of the semiring. In particular,
we may interpret $E$ over the \NEW{free commutative semiring} $\N[x_1,\ldots, x_m]$ (semiring of commutative formal polynomials in the indeterminates $x_1,\ldots,x_m$, with coefficients in the semiring of natural numbers). We
shall say that a monomial 
$x_1^{\alpha_1}\cdots x_m^{\alpha_m}$
\NEW{appears} in the expression $E$ if there exists a positive integer $c$ such that
$c x_1^{\alpha_1}\cdots x_m^{\alpha_m}$ appears in the expansion
of the polynomial obtained by interpreting $E$ in $\N[x_1,\ldots, x_m]$. The coefficient $c$ is the \NEW{multiplicity} of the monomial.
For instance, the monomials $x_1x_3$, $x_3$, and $x_2$, are the only
ones appearing in the polynomial expression above, their respective multiplicities are $3$, $3$, and $1$. 

%

\begin{definition} 
If $P,Q$ are positive polynomial expressions, we say that the identity
$P=Q$ is valid 
in a semiring $\SS$ if it holds for any substitution $x_1=s_1,\ldots,x_m=s_m$ of $s_1,\ldots, s_m\in \SS$.
\end{definition} 
To show that $P=Q$ holds in every commutative semiring, it suffices
to check that it is valid in the free commutative semiring $\N[x_1,\ldots, x_n])$, the variables $x_1,\ldots,x_m$ of the expression being interpreted
as the indeterminates of the semiring.
\begin{theorem}[\bf Transfer principle, weak form]\label{transfer-w}
Let $P,Q$ be positive polynomial expressions.
If the identity $P=Q$ holds in all commutative rings, then it also holds
in all commutative semirings.
\end{theorem}
We shall only prove the following variant, which is stronger.
\begin{theorem}[\bf Transfer principle, strong form] \label{transfer-s}
Let $P^+,P^-,Q^+,Q^-$ be positive polynomial expressions. If the identity $P^++Q^-=P^-+Q^+$ holds
in all commutative rings, and if there is no monomial appearing simultaneously in $Q^+$ and $Q^-$, then there is a positive polynomial expression $R$ such that
the identities
\[
P^+=Q^++R \text{ and } P^-=Q^-+R
\]
hold in all commutative semirings.
\end{theorem}
\begin{proof}
Since the identity $P^++Q^-=P^-+Q^+$ holds
in all commutative rings, it holds in particular 
when interpreted in $\Z[x_1,\ldots,x_m]$. 
Hence, the same monomials must appear on both sides of the equality
$P^++Q^-=P^-+Q^+$. Every monomial appearing in $Q^+$ must also appear
in $P^+$ 
with a multiplicity greater or equal to that of $Q^+$, 
otherwise, it would appear in 
$Q^-$, contradicting the assumption. 
We define $R$ to be the
positive polynomial expression equal
to the sum of all the terms $cm$, where either $m$ is a monomial appearing in $P^+$ but not in $Q^+$ and $c$ is its multiplicity in $P^+$, or $m$ is a monomial appearing both in $P^+$ and $Q^+$ and $c$ is the difference of their multiplicities. 
We have $P^+=Q^++R$. This identity, which holds in $\Z[x_1,\ldots,x_m]$,
holds a fortiori in $\N[x_1,\ldots,x_m]$, which is the free
commutative semiring in the indeterminates $x_1,\ldots, x_m$,
and so, it holds in every commutative semiring. 
 A symmetrical argument shows that there
is a positive polynomial expression $S$ such that the identity
$P^-=Q^-+S$ holds in all commutative semirings.
Since $\Z[x_1,\ldots,x_m]$ is cancellative, we deduce
from the identity $Q^++R+Q^-=Q^-+S+Q^+$ valid in $\Z[x_1,\ldots,x_m]$
that the identity $R=S$ is still valid in $\Z[x_1,\ldots,x_m]$.
Hence, the identities $P^+=Q^++R$ and $P^-=Q^-+R$ hold in all
semirings.
\end{proof}
The proof of the previous principle may look tautological, however,
we next show that this principle yields (by a direct application) several identities some of which have been proved in the literature by involved combinatorial arguments.

\subsection{Determinantal identities}
Gondran and Minoux~\cite{GM2}
introduced the following general notion of bideterminant, which
applies to matrices with entries in an arbitrary semiring.
\begin{definition} \label{Ddet}
The \NEW{bideterminant} of $A=[a_{ij}]\in\M(\SS)$ is the pair $(|A|^+, | A| ^-)$, where
\e
\label{E2}
| A| ^+=\sum\limits_{\sigma\in \permeven_n} (a_{1\sigma(1)}\cdots  a_{n\sigma(n)})  , \quad
| A| ^-=\sum\limits_{\sigma\in \allperm_n \setminus \permeven_n} (a_{1\sigma(1)}\cdots  a_{n\sigma(n)} ) 
\ee
where 
$\allperm_n$ is the permutation group on $n$ elements and
$\permeven_n\subset \allperm_n$ is the subgroup of the even permutations.
\end{definition}

\begin{example}[Multiplicativity for the determinant] 
We consider the multiplicativity of the determinant, i.e. the identity $\detp{AB}=\detp{A}\detp{B}$, where $\detp{A}$ denotes the determinant of $A$,
 which is valid for $n\times n$ matrices $A$, $B$ with entries
in a commutative ring. 
In the introduced notations it can be re-written via the positive polynomial expressions in these variables as follows:
$$|AB|^+ -|AB|^-= |A|^+ |B|^+ + |A|^- |B|^- -|A|^+ |B|^- - |A|^- |B|^+,$$
or
\begin{align}
|AB|^+ + |A|^+ |B|^- + |A|^- |B|^+ =|AB| ^- + |A|^+ |B|^+ + |A|^- |B|^-
\enspace .\label{eq-weakmult}
\end{align}

1. The weak transfer principle shows that the previous
identity is valid in any commutative semiring.

2. The strong transfer principle shows that there is an element
$s$ of the semiring (which is given by a positive
polynomial expression in the entries of $A$ and $B$) such that:
\begin{subequations}\label{e-smul}
\begin{gather}
|AB|^+  = |A|^+ |B|^+ + |A|^-  |B|^- + s \label{e-smul1}\\
|AB|^- = |A|^+ |B|^- + |A|^-  |B|^+  + s\label{e-smul2}
\end{gather}
\end{subequations}
\end{example}

\begin{example}[Binet-Cauchy formula] \label{ex-binet-chauchy}
 We denote by $Q_{r,k}$ the set of all ordered sequences $(i_1,\ldots, i_r)$, where $1\le i_1 <i_2<\ldots < i_r\le k$. If $\alpha = (i_1,\ldots, i_r)\in Q_{r,k}$, $\beta=(j_1,\ldots,j_s)\in Q_{s,l}$, $X\in {\cal M}_{k,l}(\SS)$, where $\SS$ is a semiring, then $X[\alpha\vert \beta]$ denotes the $r\times s$ submatrix of $X$ located in the intersection of the rows numbered $i_1,\ldots, i_r$ and columns numbered $j_1,\ldots, j_s$.

Let $\RR$ be a commutative ring, $A\in {\cal M}_{n,p}(\RR)$, $B \in {\cal M}_{p,m}(\RR)$, $C:=AB \in {\cal M}_{n,m}(\RR)$. The Binet-Cauchy theorem states that for any $r$, $1\le r\le \min\{ n,m,p\}$ and for any $\alpha\in Q_{r,n}$, $\beta\in Q_{r,m}$ the following formula holds:
$$ \detp{C[\alpha\vert \beta]} = \sum\limits_{\omega \in Q_{r,p}} \detp{A[\alpha\vert \omega]}\detp {B[\omega\vert\beta]}.$$

Using the weak transfer principle, we get that for matrices  $A\in {\cal M}_{n,p}({\cal S})$, $B \in {\cal M}_{p,m}({\cal S})$, $C:=AB \in {\cal M}_{n,m}({\cal S})$ over an arbitrary commutative semiring $\SS$ the following identity holds:
\begin{gather*}
 |(C[\alpha\vert \beta])|^+  +  \sum\limits_{\omega \in Q_{r,p}} ( |A[\alpha\vert \omega]|^+B[\omega\vert\beta]|^- + |A[\alpha\vert \omega]|^-B[\omega\vert\beta]|^+)\\
\quad =|(C[\alpha\vert \beta])|^- + \sum\limits_{\omega \in Q_{r,p}} (|A[\alpha\vert \omega]|^+B[\omega\vert\beta]|^++ |A[\alpha\vert \omega]|^-B[\omega\vert\beta]|^-).\end{gather*}
The strong transfer principle implies that there exists an element $s\in {\cal S}$
such that
\begin{align*}
 |(C[\alpha\vert \beta])|^+  &=   \sum\limits_{\omega \in Q_{r,p}} (|A[\alpha\vert \omega]|^+B[\omega\vert\beta]|^++ |A[\alpha\vert \omega]|^-B[\omega\vert\beta]|^-) +s ,\\
 |(C[\alpha\vert \beta])|^-  &=\sum\limits_{\omega \in Q_{r,p}} ( |A[\alpha\vert \omega]|^+B[\omega\vert\beta]|^- + |A[\alpha\vert \omega]|^-B[\omega\vert\beta]|^+) +s.\end{align*}

The latter identity was stated by Butkovi\v{c}, Cuninghame-Green and Gaubert \cite{GBCG}, the former was stated by Poplin and Hartwig, see~\cite[Theorem 5.4]{poplinhartwig}. 
\end{example}
\begin{example}[Cramer identity]\label{ex-cramer}
If $A$ is a $n\times n$ matrix with entries in a commutative semiring, we denote
by $A(i,j)$ the $(n-1)\times (n-1)$ submatrix in which row $i$ and column $j$
have been suppressed. Define
the {\em positive adjoint matrix} of $A$ to be the $n\times n$ matrix $A\adjp$ with $(i,j)$-entry:
\[
(A\adjp)_{ij}:= \begin{cases} |A(j,i)|^+ & \text{ if $i+j$ is even}\\
|A(j,i)|^- & \text{ if $i+j$ is odd.}\end{cases}
\]
The {\em negative adjoint matrix} $A\adjm$ is defined similarly, by exchanging
the parity condition on $i+j$. When the entries of $A$ belong to a commutative ring, the classical adjoint matrix $A\adj$ is such that $A\adj=A\adjp-A\adjm$,
and we have
\[
|A| I = AA\adj = A\adj A \enspace ,
\]
where $I$ is the identity matrix.
The weak transfer principle applied to the first identity
shows that
\[
|A|^+I + AA\adjm = |A|^-I+AA\adjp 
\enspace .
\]
The strong transfer principle implies that there exists a matrix $R$
such that
\[
AA\adjp =|A|^+I + R,\qquad AA\adjm =|A|^-I + R \enspace ,
\]
a result which was used by Reutenauer and Straubing~\cite[Lemma3]{reutstraub}.
\end{example}

\begin{example}[Cayley-Hamilton formula] 
Let $A$ be a $n\times n$ matrix.
For $1\leq k\leq n$, let $\Lambda^k_\pm(A)$ denote the $k$-th positive
or negative compound matrix of $A$, which is the 
$\binom{n}{k}\times \binom{n}{k}$ matrix
indexed by the nonempty subsets $I,J$ of $k$ elements of $\{1,\ldots,n\}$ such
that $(\Lambda^k_\pm(A))_{IJ}:=|A[I\vert J]|^\pm$. The classical $k$-th compound
matrix is $\Lambda^k(A):=\Lambda^k_+(A)-\Lambda^k_-(A)$.
The characteristic polynomial
of $A$ can be written as $P_A(y)=\detp{A-y I}=(-y)^n+ \sum_{1\leq k\leq n}
(-y)^{n-k}\operatorname{tr}\Lambda^k(A)$ where $\operatorname{tr}$ denotes
the trace of a matrix. The Cayley-Hamilton theorem
shows that the identity
\[
(-A)^n+ \sum_{1\leq k\leq n} (-A)^{n-k}\operatorname{tr}\Lambda^k(A) =0 
\]
is valid in any commutative ring. Hence, the weak transfer principle
shows that the identity
\begin{gather}
A^n+ 
\sum_{1\leq k\leq n\atop k\text{ even } } A^{n-k}\operatorname{tr}(\Lambda_+^k(A)) 
+ 
\sum_{1\leq k\leq n\atop k\text{ odd } } A^{n-k}\operatorname{tr}(\Lambda_-^k(A)) 
\nonumber\\
=
\sum_{1\leq k\leq n\atop k\text{ even } } A^{n-k}\operatorname{tr}(\Lambda_-^k(A)) 
+ 
\sum_{1\leq k\leq n\atop k\text{ odd } } A^{n-k}\operatorname{tr}(\Lambda_+^k(A)) 
\label{e-CH}
\end{gather}
holds in any commutative semiring.
This result was first proved combinatorially by Straubing~\cite{straubing}.
\end{example}
The semiring version of the Cayley-Hamilton theorem is weaker
than the ring version, however, it still has useful consequences,
as in the following application.
We say that a sequence $s_0,s_1,\ldots$ of elements
of a semiring is {\em linear recurrent} with a representation of dimension $n$
if $s_k=cA^kb$ for all $k$, where $c,A,b$ are matrices with entries
in the semiring, of respective sizes $1\times n$, $n\times n$, and $n\times 1$.
Left and right multiplying the identity~\eqref{e-CH} by $cA^p$ and $b$,
respectively, we see that
\begin{gather*}
s_{n+p}
+
\sum_{1\leq k\leq n\atop k\text{ even } } s_{n+p-k}\operatorname{tr}(\Lambda_+^k(A)) 
+ 
\sum_{1\leq k\leq n\atop k\text{ odd } } s_{n+p-k}\operatorname{tr}(\Lambda_-^k(A)) \\
=
\sum_{1\leq k\leq n\atop k\text{ even } } s_{n+p-k}\operatorname{tr}(\Lambda_-^k(A)) 
+ 
\sum_{1\leq k\leq n\atop k\text{ odd } } s_{n+p-k}\operatorname{tr}(\Lambda_+^k(A)) .
\end{gather*}
Hence, an immediate induction shows that {\em a linear recurrent sequence with a representation of dimension $n$ is identically zero as soon as its first $n$ coefficients are zero}.
\subsection{Polynomial identities for matrices}

In the following three examples we give the semiring versions of three classical results in PI-theory.

We note that the matrix algebra over a field is a PI-algebra, i.e., it satisfies non-trivial polynomial identities, since it is finite-dimensional.

\begin{example}[Amitsur-Levitzki's identity] 
The famous Amitsur-Levitzki
theorem states that the minimal (by the degree) polynomial identity for
the algebra of $n\times n$ matrices
over any commutative ring is 
$$S_{2n}(x_1,\dots,x_{2n}) =0,$$
where, for all $n$, $S_n$ denotes the polynomial
\begin{equation}\label{defsym}
S_n(x_1,\dots,x_{n}):=\sum_{ \sigma\in 
\allperm_{n}}\sgn(\sigma) x_{\sigma(1)}\cdots x_{\sigma(n)},
\end{equation}
and, for any permutation $\sigma$,  $\sgn(\sigma)$ denotes its signature.

Hence, weak transfer principle provides that the equality
$$\sum_{ \sigma\in \permeven_{2n}} x_{\sigma(1)}\cdots x_{\sigma(2n)} =\sum_{ \sigma\in \allperm_{2n}\setminus \permeven_{2n}}x_{\sigma(1)}\cdots x_{\sigma(2n)}$$
is a polynomial identity for matrices over any commutative semiring. The semiring version of Amitsur-Levitzki
theorem  was firstly stated and proved in \cite[Lemma 7.1]{gaubertburnside}, where it was used for the positive solution of the Burnside problem for semigroups of matrices over a class of commutative idempotent semirings. 

Also, the strong transfer principle implies that for any subset $S'\subseteq \allperm_{2n}$ there exists a matrix polynomial $R=R(S')$ such that
$$\sum_{ \sigma\in \permeven_{2n}\cap S'} x_{\sigma(1)}\cdots x_{\sigma(2n)} =\sum_{ \sigma\in \allperm_{2n}\setminus (\permeven_{2n}\cup S')}x_{\sigma(1)}\cdots x_{\sigma(2n)}+R$$
and 
$$\sum_{ \sigma\in (\allperm_{2n}\setminus \permeven_{2n})\cap S'} x_{\sigma(1)}\cdots x_{\sigma(2n)} =\sum_{ \sigma\in \permeven_{2n}\setminus S'} x_{\sigma(1)}\cdots x_{\sigma(2n)} +R$$
hold, namely we can take
$$Q_+=\sum_{ \sigma\in \allperm_{2n}\setminus (\permeven_{2n}\cup S')}x_{\sigma(1)}\cdots x_{\sigma(2n)}\mbox{ and } Q_-=\sum_{ \sigma\in \permeven_{2n}\setminus S'} x_{\sigma(1)}\cdots x_{\sigma(2n)},
$$
since for any $S'$ the set $(\allperm_{2n}\setminus (\permeven_{2n}\cup S'))\cap (\permeven_{2n}\setminus S')=\emptyset$, i.e., there is no monomials appearing simultaneously in $Q^+$ and $Q^-$.
\end{example}

\begin{example}[Capelly identity] The identity
$$K_n(x_1,\ldots, x_n,y_1,\ldots,y_{n+1}) := \sum\limits_{\sigma\in \allperm_n} (-1)^\sigma y_1x_{\sigma(1)}y_2x_{\sigma(2)} \cdots y_n x_{\sigma(n)}y_{n+1} =0$$ 
holds for matrices over any commutative ring.

Hence, the weak transfer principle implies that the identity
$$ \sum\limits_{\sigma\in \permeven_n}  y_1x_{\sigma(1)}y_2x_{\sigma(2)} \cdots y_n x_{\sigma(n)}y_{n+1} = \sum\limits_{\sigma\in \allperm_n\setminus \permeven_n}  y_1x_{\sigma(1)}y_2x_{\sigma(2)} \cdots y_n x_{\sigma(n)}y_{n+1}$$ 
holds in any commutative semiring.

The strong transfer principle gives that there exists a matrix polynomial $R$ such that
\begin{gather*}
\sum\limits_{\sigma\in S' \cap \permeven_n}   y_1x_{\sigma(1)}y_2x_{\sigma(2)} \cdots y_n x_{\sigma(n)}y_{n+1} \\
= \sum\limits_{\sigma\in \allperm_{n}\setminus (\permeven_{n}\cup S')} y_1x_{\sigma(1)}y_2x_{\sigma(2)} \cdots y_n x_{\sigma(n)}y_{n+1}+R
\end{gather*}
and \begin{gather*}
 \sum\limits_{\sigma\in (\allperm_{n}\setminus \permeven_{n})\cap S'}  y_1x_{\sigma(1)}y_2x_{\sigma(2)} \cdots y_n x_{\sigma(n)}y_{n+1} \\
= \sum\limits_{\sigma\in \permeven_{n}\setminus S'}  y_1x_{\sigma(1)}y_2x_{\sigma(2)} \cdots y_n x_{\sigma(n)}y_{n+1}+R
\end{gather*}
are polynomial identities for any subset $S'\subseteq \allperm_n$, here $R$ depends on $S'$. 
\end{example}
\begin{example}[Identity of algebraicity] 
The identity
$$
{\cal A}(y,z):=S_{n^2}([y^{n^2},z],\ldots,[y,z])=0
$$
where the polynomial $S$ is as in~\eqref{defsym} and $[y,z]:=yz-zy$,
holds for matrices over any commutative ring.

Hence, the weak transfer principle implies that the identity ${\cal A}^+(y,z)={\cal A}^-(y,z)$ holds in any commutative semiring. Here ${\cal A}^+(y,z)$ denotes the sum of monomials of ${\cal A}(y,z)$ which go with the positive sign, and ${\cal A}^-(y,z)$ denotes the sum of monomials of ${\cal A}(y,z)$ which go with the sign ``$-$''.
The strong transfer principle is not applicable here since cancellations appear in ${\cal A}$, so the condition that there are no equal monomials may not be satisfied.
\end{example}

\section{Semirings with a symmetry}\label{sec4}

\begin{definition}
A map $\tau:\SS\to \SS$ is a symmetry if $\tau$ is a left and right $\SS$-semimodule homomorphism from $\SS$ to $\SS$ of order 2, i.e.,
\begin{subequations}
\begin{align}
& \tau(a+b)=\tau(a)+\tau(b) \\ 
& \tau(0)=0\\ 
& \tau(a\cdot b)=a\cdot\tau(b)= \tau(a)\cdot b \\
& \tau(\tau(a))=a .
\end{align}
\end{subequations}
\end{definition}
\begin{ex}
A trivial example of symmetry is $\tau(a)=a$.
Of course, in a ring, we may take $\tau(a)=-a$. 
\end{ex}
In the sequel,
in a general semiring with symmetry, we will write $-a$ instead of $\tau(a)$,
and $a-b$ for $a+(-b)=a+\tau(b)$, understanding that $a-a=a+\tau(a)$ may be different from zero.
Also we may use the notation $+a$ instead of $a$. 

\begin{definition}\label{def-acirc}
A map $f:\SS\to\SS'$ between semirings with symmetry is a \NEW{morphism of semirings with symmetry} if $f$ is a morphism of semirings such that $f(-s)=-f(s)$ for all $s\in\SS$.
\end{definition}

\begin{definition}
For any $a\in \SS$, we set $a^\circ:=a-a$,
so that $-a^\circ=a^\circ=(-a)^\circ$, and we denote
\[
\SS^\circ := \{a^\circ \mid a\in \SS\}  ,\qquad  \SS^\vee:=(\SS\setminus \SS^\circ )\cup\{\oo\}\enspace .
\]
\end{definition}
\begin{re}
The set $\SS^\circ$ is a left and right ideal of $\SS$. 
\end{re}
\begin{definition}
We define the \NEW{balance} relation $\balance$ on $\SS$, by 
\[ a \balance b\iff a-b\in \SS^\circ \enspace .
\]
\end{definition}
\begin{re}
The relation $\balance$ is reflexive and symmetric, but we shall see
in Section~\ref{S5.1} that it may not be transitive.
\end{re}
Observe that
\[
a-b\balance c \iff a\balance b+c \enspace .
\]
\begin{definition}
We introduce the following relation:
\[
a\succeq^\circ b \iff a=b+c \text{ for some } c\in \SS^\circ \enspace .
\]
\end{definition}
\begin{re}
This relation is reflexive and transitive. It may not be 
antisymmetric, see Example~\ref{ex_not_antisym} below.
\end{re}
\begin{re} If $a=b+c$ with $c\in\SS^\circ$ then $a-b=b^\circ+c\in\SS^\circ$, 
hence 
\begin{align}\label{succeq-bal}
 a\succeq^\circ b \text{ or } b\succeq^\circ a \Rightarrow a\balance b
\enspace.
\end{align}
The converse is false in general.
Indeed, let $\SS$ be the semiring $\R^2$ with the entrywise laws,
and the symmetry $\tau(a)=a$. Then $\SS^\circ=\SS$, hence $a\balance b$ holds
for all $a,b\in\SS$, whereas $a\succeq^\circ b$ is equivalent to $a\geq b$.
Since $(1,2)$ and $(2,1)$ are not comparable in $\SS$, this contradicts
the converse implication in~\eqref{succeq-bal}.
\end{re}

We shall also apply the notation $\balance$ and $\succeq^\circ$ to matrices
and vectors, understanding that the relation holds entrywise. 

\begin{ex}\label{semiring-couples}
Let $\SS$ denote an arbitrary semiring. An interesting semiring with symmetry is the set of couples $\SS^2$ equipped with
the laws:
\begin{gather*}
(x',x'')+ (y',y'') = (x'+ y', x''+ y''),\\
(x',x'')\cdot  (y',y'') = (x'\cdot  y' + x''\cdot  y'', x'\cdot  y''+ x''\cdot  y'),\\
-(x',x'')=(x'',x') \enspace .
\end{gather*}
 The zero and unit of $\SS^2$ are $(0,0)$ and $(1,0)$.
The map $x'\mapsto (x',0)$ is an embedding from $\SS$ to $\SS^2$,
which allow us to write $x'$ or $+x'$ instead of $(x',0)$,
$-x''$ instead of $(0,x'')$, and
$x'-x''$ instead of $(x',x'')$. 
Let us define the \NEW{modulus}, $\mymod{x}$, of an element $x=(x',x'')\in \SS^2$ 
to be $\mymod{x}:=x' + x''$.
Then, $x^\circ=(\mymod{x},\mymod{x})$ for all $x\in\SS^2$ and  the map $x\mapsto \mymod{x}$ is a surjective morphism from $\SS^2$ to $\SS$. Hence 
the elements of $(\SS^2)^\circ$ are the couples of the form $(x',x')$ 
with $x'\in\SS$, and we have $(x',x'')\balance (y',y'')\Leftrightarrow 
x'+y''=x''+y'$.
Finally, if $\SS$ is already a semiring with symmetry, the map
$\pi:\SS^2\to\SS$ such that $\pi((x',x''))=x'-x''$ is a surjective morphism of semirings with symmetry.
\end{ex}

Now we can give an example showing that $\succeq^\circ$ is not anti-symmetric.
\begin{ex} \label{ex_not_antisym}
Let $\SS=\Z$ be the ring of integers. We consider $\SS^2$ with 
the same laws as in the previous example. Then
$$ (1,2)\succeq^\circ (0,1)$$
since $(1,2)=(0,1)+(1,1)$
and $$ (0,1)\succeq^\circ (1,2)$$
since $(0,1)=(1,2)+(-1,-1)$,
however, $(0,1)\ne (1,2)$.
\end{ex}

Other examples of semirings with symmetry
shall be given in the next section. 

\begin{re} \label{re-comb-id}
The combinatorial identities of Section~\ref{sec-combinat} 
can be rewritten in a more familiar way by working in the semiring 
with symmetry $\SS^2$ defined in Example~\ref{semiring-couples},
that is by identifying $\SS$ as a subsemiring of $\SS^2$.\end{re}

\bigskip

In particular, the above notations allow us to define the determinant of
matrices as follows.

\begin{definition} \label{de:detS}
Let $\SS$ be a semiring with symmetry and $A=[a_{ij}]\in\M(\SS)$.
We define the determinant $\detp{A}$ of $A$ to be the element
of $\SS$ defined by the usual formula 
$$\sum_{\sigma\in \allperm_n} \sgn(\sigma) a_{1\sigma(1)}\cdots a_{n\sigma(n)} ,$$ 
understanding that $\sgn(\sigma)=\pm 1$ depending on the even or odd parity of $\sigma$. 
\end{definition}
\begin{re}
With this definition, we have that $\detp{A}=|A|^+-|A|^-$,
where each of $|A|^+$ and $|A|^-$ are as in Definition~\ref{Ddet}.
\end{re}
\begin{re}\label{rk-def-perm}
If the symmetry of the semiring $\SS$ is the identity map, i.e., if $-a:=a$,
then, the determinant $\detp{A}$ coincides with the permanent
of $A$:
\e
\label{E1}
\per(A):=\sum\limits_{\sigma\in \allperm_n} a_{1\sigma(1)}\cdots  a_{n\sigma(n)}   
\ee 
where $\allperm_n$ is the permutation group on the set $\{1,\ldots,n\}$.
\end{re}

\begin{re}\label{def-dett}
Identifying any semiring $\SS$ (not necessarily with symmetry) as a subsemiring of $\SS^2$, we may define the determinant of any square matrix $A$ with entries in $\SS$ as its determinant as a matrix with values in $\SS^2$. This quantity that we shall denote by $\dett{A}$ is nothing but the bideterminant of $A$, that is $(|A|^+,|A|^-)$ of $\SS^2$, where each of $|A|^+$ and $|A|^-$ are as in Definition~\ref{Ddet}. If $\SS$ is a semiring with symmetry, $\dett{A}$ does not coincide in general with $\detp{A}$ since $|A|^-$ is in general different from $0$.
\end{re}

The results of the previous section can be reformulated in the following way.
\begin{corollary}
Let $\SS$ be an arbitrary semiring, 
$A,B\in \M(\SS)$, and $\dett{\cdot}$ be defined as in Remark~\ref{def-dett}.
The weak form of the multiplicative property of the determinant~\eqref{eq-weakmult} can be rewritten
equivalently as
\begin{align}\label{e-mult3}
\dett{AB}\balance \dett{A} \dett{B} \enspace .
\end{align}
The strong form~\eqref{e-smul} yields
\begin{align}\label{e-mult4}
\dett{AB}\succeq^\circ \dett{A} \dett{B} \enspace .
\end{align}
\end{corollary}

\begin{corollary}
The Cayley-Hamilton theorem can be
rewritten as $P_A(A)\balance 0$ for $A\in \M(\SS)$,
where $P_A(y)=\dett{A-yI}$.
\end{corollary}

More generally, if $\SS$ is a semiring with symmetry, every
combinatorial identity can be expressed in $\SS$ in the usual form,
by replacing the equality by the $\balance$ or the $\succeq^\circ$ symbol.
For instance,
the relations~\eqref{e-mult3} and~\eqref{e-mult4} hold not only 
for matrices with entries in $\SS$ viewed as matrices with entries in $\SS^2$
but also in an arbitrary semiring with symmetry $\SS$, replacing 
the determinant function $\dett{\cdot}$ in $\SS^2$ by the determinant
function $\detp{\cdot}$ of Definition~\ref{de:detS} in $\SS$.

Indeed, let us say that $P$ is a \NEW{polynomial expression} if it is the 
formal difference $P=P^+-P^-$ of positive polynomial expressions, 
and interpret it in an arbitrary semiring with symmetry by understanding the symbol $-$ as the symmetry of the semiring.
Then, considering in any ring the symmetry $a\mapsto -a$ where $-a$ is the opposite of $a$ for the additive law,  Theorems~\ref{transfer-w} and \ref{transfer-s} can be rewritten in the following equivalent manner.

\begin{theorem}[\bf Transfer principle, weak form]\label{transfer-w-2}
Let $P$ and $Q$ be polynomial expressions.
If the identity $P=Q$ holds in all commutative rings, then 
the identity $P\balance Q$ holds in all commutative semirings with symmetry.
\end{theorem}
\begin{theorem}[\bf Transfer principle, strong form] \label{transfer-s-2}
Let $P$ and $Q$ be polynomial expressions. 
If the identity $P=Q$ holds in all commutative rings, and if $Q=Q^+-Q^-$
for some positive polynomial expressions such that there is no monomial appearing simultaneously in $Q^+$ and $Q^-$, then the identity
\[ P \succeq^\circ Q \]
holds in all commutative semirings with symmetry.
\end{theorem}

\section{Extensions of the max-plus semiring}
We next present two related extensions of the max-plus semiring.
The first one, the \NEW{symmetrized max-plus semiring},  was introduced by M.~Plus~\cite{Plus} (see also~\cite{BCOQ92}). 
The second one was introduced by Izhakian~\cite{Izh05,Izh} to study 
linear independence in the tropical sense. We shall see 
that both semirings can be obtained by a more general construction,
which encompasses other interesting examples of semirings
like the semiring of ``jets'' used in~\cite{finkelstein} and~\cite{Bapat}.
\begin{prop}
Let $(\SS,+,\cdot)$ be a semiring. Then the set $\SS\times \R$ endowed 
with the operations
$$
(a,b)\oplus(a',b')=\left\{\begin{array}{lcl}
(a+a',b) & \text{ if } & b=b' \\
(a,b) & \text{ if } & b>b' \\
(a',b') & \text{ if } & b<b' \end{array} \right.
$$
and 
$$
(a,b)\odot (a',b')=(a \cdot a',b\odot  b')
$$
is a semiring, with zero $(\oo,\ooo)$ and unit $(\oi,\ooi)$.
If $\SS$ is a zero-sum free semiring without zero divisors, then
the set:
$$\SS\R:=(\SS\setminus \{\oo\})\times (\R\setminus\{\ooo\}) \cup 
\{(\oo,\ooo)\}$$
is a subsemiring of $\SS\times \R$.
\end{prop}
We shall denote by $\zero$ and $\unit$, instead of $(\oo,\ooo)$ and
$(\oi,\ooi)$, the zero and unit of $\SS\times \rmax$.
\begin{re}\label{extension-of-rmax}
Let us define the \NEW{modulus} of $x=(a,b)\in \SS\times \R$ by $\mymod{x}:=b$.
It is clear that the modulus map $x\mapsto \mymod{x}$ is a surjective morphism from 
$\SS\times \R$ to $\R$. Moreover,
the maps $a\mapsto (a,\ooo)$ and $a\mapsto (a,\ooi)$
are embeddings from $\SS$ to $\SS\times \R$, and the map $b\mapsto (\oo,b)$ is
 an embedding from $\R$ to $\SS\times \R$. 
When  $\SS$ is zero-sum free without zero divisors, the modulus map
is also a surjective morphism from  $\SS\R$ to $\R$,
and the  map 
$\SS\to\SS\R$ which sends $a\in\SS\setminus \{\oo\}$ to $(a,\ooi)$ and
$\oo$ to $\zero$ is an embedding.
However, $\R$ is not necessarily embedded in $\SS\R$, because
the natural injection which sends $b\in\R\setminus  \{\ooo\}$ to $(\oi,b)$
and $\ooo$  to $\zero$ is not a morphism unless $\SS$ is
idempotent. 
If $\SS$ is a semiring with symmetry, then the map
$(a,b)\mapsto (-a,b)$ is a symmetry on $\SS\times \R$ or $\SS\R$,
and we have $(\SS\times\R)^\circ=\SS^\circ\times\R$ and 
$(\SS\R)^\circ=(\SS^\circ\setminus \{\oo\})\times (\R\setminus\{\ooo\}) \cup 
\{\zero\}$.
We shall call $\SS\R$ an \NEW{extension} of $\R$.
\end{re}

We next show that several important semirings can be obtained as extensions of $\R$.


\subsection{The symmetrized max-plus semiring}\label{S5.1}
The symmetrized max-plus semiring, which is useful to deal with 
systems of linear equations over $\R$, was introduced in~\cite{Plus}. We recall here some definitions and results from~\cite{Plus}, 
and show that this semiring can also be obtained by the general construction of the previous section.

Consider the semiring with symmetry $\R^2$ defined as 
in Example~\ref{semiring-couples}, using the notations $\oplus$, $\odot$ and $\ominus$ instead of $+$, $\cdot$ and $-$ (for instance $\ominus (x',x'')=(x'',x')$) and let us use the notations
$a^\circ$, $\mymod{\cdot}$, $\balance$ and $\succeq^\circ$ as in Section~\ref{sec4}.
The classical way to obtain the ring of integers $\Z$ 
is by a ``symmetrization'' of the semiring of nonnegative integers 
$\N$, which is obtained by quotienting the semiring with symmetry
$\N^2$ by the relation $\balance$. The same cannot be done
when replacing $\N$ by the max-plus semiring $\rmax$,
because the relation $\balance$ is not transitive in $\rmax^2$.
Indeed, $(\ooi,1)\balance (1,1)$ and $(1,1)\balance (1,\ooi)$, but $(\ooi,1)\notbalance (1,\ooi)$.

Instead of $\balance$, we shall consider the following thinner relation.
\begin{definition} 
The relation ${\cal R}$ on $\R^2$ is defined by:
$$ (x',x'')\: {\cal R} \: (y',y'')  \Longleftrightarrow \left\{ \begin{array}{l} x'\ne x'', y'\ne y'' \text{ and } x'\oplus y'' = x''\oplus y'   \\ \text{or} \\ x'=x''=y'=y'' \end{array} \right.$$
\end{definition}
\begin{definition} 
With a given $a\in\R^2$ we associate the following subset in~$\R^2$:
$$\operatorname{Sol}(a):=\{x\in \R^2\vert x\balance a\}\enspace.$$
\end{definition}
\begin{re}
It can be checked that
\[ a{\cal R} b \iff \operatorname{Sol}(a)=\operatorname{Sol}(b)
\enspace .
\]
\end{re}
It follows that ${\cal R}$ is an equivalence relation on $\R^2$.
The relation ${\cal R}$ is compatible with the relations
or operations $\balance$, $\succeq^\circ$, $\mymod{\cdot }$, $a\mapsto a^\circ$, $\ominus$, $\oplus$, and $\odot$ on $\R^2$. 

Therefore the following quotient semiring can be considered:
\begin{definition}[{\cite{Plus}}]\label{def-sym-mplus}
The \NEW{symmetrized max-plus semiring} is $\S:=\R^2/{\cal R}$. 
We denote the induced operations on $\S$ by the same notations as in
$\R^2$: $\oplus, \odot , \balance$, etc.
\end{definition}

The elements of $\S$ are the classes 
$\overline{(t,\ooo)}=\set{(t,x'')}{x''<t}$, 
$\overline{(\ooo,t)}=\set{(x',t)}{\\ x' < t}$, and
 $\overline{(t,t)}=\{(t,t)\}$,  for $t\in\R\setminus\{\ooo\}$, and 
the class $\overline{(\ooo,\ooo)}$. 

%
\begin{definition}
Let $\overline x=\overline{(x',x'')}\in \S$. Then $\overline x$ is called \NEW{sign-positive} (resp.\ \NEW{sign-negative}) if either $x'>x''$ (resp.\ $x''>x'$) or $x'=x''=\ooo$ for a representation of the class $\overline x$. The element $\overline x\in \S$ is called \NEW{signed} if it is either sign-negative or sign-positive, $\overline x$ is called \NEW{balanced} if $x'=x''$ for any representation of the class $\overline x$, otherwise it is called \NEW{unbalanced}. 
\end{definition} 
The sets of sign-positive, sign-negative and balanced elements are denoted respectively by $\smax^\oplus$, $\smax^\ominus$, and $\smax^{\circ}$. Note that $\S=\smax^\oplus \cup \smax^\ominus \cup \smax^\circ$. We denote $\smax^\vee=\smax^\oplus \cup \smax^\ominus$. The notation $\smax^\circ$ and $\smax^\vee$ is justified, because
$(\smax)^\circ=\set{a^\circ}{a\in \smax}$, as in Definition~\ref{def-acirc},
coincides with $\smax^\circ$.

The following result is immediate.
\begin{lemma}
The embedding $t\mapsto \overline{(t,\ooo)}$ sends $\R$ to $\smax^\oplus$.
Moreover, $\smax^\ominus=\ominus(\smax^\oplus)$ and
$\smax^\circ=(\smax^\oplus)^\circ$. 
\qed
\end{lemma}

Thus we can identify sign-positive elements of $\S$ with the elements of $\R$, and we can write sign-negative (resp.\ balanced) elements of $\S$ as $\ominus x$
(resp.\ $x^\circ$) with $x\in \R$. We will do this without further notice if it will not lead to a misunderstanding. So, $x^\circ =\overline{(x,\ooo)}\:^\circ=\overline{(x,x)}$. In these notations the subtractivity rules in $\S$ look as follows:
$$ \begin{array}{ll}
a \ominus b =a & \mbox{ if } a>b \\
a \ominus b =\ominus b & \mbox{ if } a<b \\
a\ominus a =a^\circ
\end{array}
$$

The Boolean semiring $\B$ is the subsemiring of $\R$ composed of the 
neutral elements $\ooo$ and $\ooi$.
Since the relation ${\cal R}$ is trivial on $\B^2$, the quotient of $\B^2$ over this relation does not glue anything, and in appropriate notations we have the following.

\begin{definition}\label{def-symb} The \NEW{symmetrized Boolean semiring} is the 
subsemiring $\BB:=\{\ooo, \ooi,\ominus \ooi, \ooi^\circ\}$ of $\S$:

\begin{center}
\setlength{\unitlength}{0.6pt}%
\begin{picture}(100,110)
\put(45,102){$\ooi^\circ$}
\put(45,100){\line (-1,-1){38}}
\put(55,100){\line (1,-1){38}}
\put(1,48){$\ooi$}
\put(90,48){$\ominus \ooi$}
\put(5,43){\line (1,-1){38}}
\put(95,43){\line (-1,-1){38}}
\put(45,0){$\ooo$}
\end{picture}
\end{center}
\end{definition}

It is straightforward to see that the extension $\BB\R$
is isomorphic to $\S$ by the map $\BB\R\to\S$,
$(\ooo,\ooo)\mapsto\ooo=\overline{(\ooo,\ooo)}$,
$(\ooi,t)\mapsto t=\overline{(t,\ooo)}$, $(\ominus \ooi,t)\mapsto 
\ominus t=\overline{(\ooo,t)}$, $(\ooi^\circ,t)\mapsto t^\circ=\overline{(t,t)}$,
for $t\in\R\setminus\{\ooo\}$. Moreover this map is an isomorphism of semirings with symmetry, hence the operations and relation  $\ominus$,
$a\mapsto a^\circ$, $\balance$ and $\succeq^\circ$ are identical in both
representations. Finally, the modulus
maps on $\S$ and $\BB\R$ coincide.

\subsection{Izhakian's extension of the max-plus semiring} 

\begin{definition}
Let $\N_{q}$ denote the semiring which is the quotient of the semiring $\N$ of nonnegative integers by the equivalence relation, which identifies $q$ with $q+1$, $q+2$, $\ldots$. For example $2+k=\ldots=2+1=2$ in~$\N_2$
for all $k\ge 0$.
\end{definition}

\begin{definition}
The \NEW{extended tropical semiring} is the extension $\Ti:=\N_2\R$
of $\R$.
\end{definition}
This semiring structure encodes whether the maximum is attained at least two
times in an expression.
\begin{re}
In~\cite{Izh05,Izh}, Izhakian introduced the extended tropical semiring 
by equipping the set $\Re\cup \Ub\cup\{-\infty\}$, 
where $\Ub$ is another copy of $\Re$, with laws $\oplus,\odot$
defined by explicit formula, according to the membership
of the arguments of the laws to one of the three sets $\Re$, $\Ub$, and $\{-\infty\}$.
The elements of $\Ub$ are denoted by 
$a^{\nu}$ with $a\in\Re$. For instance, $2\oplus 3^\nu=3^\nu$, $3\oplus 3=3^\nu$, and $2^\nu\oplus 3=3$. 
One can check that the map $\Ti \to \Re\cup \Ub\cup\{-\infty\}$, 
sending $\zero$ to $-\infty$, $(1,a)$ to $a$, and $(2,a)$ to $a^{\nu}$ for
$a\in\Re$ is an isomorphism.
\end{re}

\begin{re}\label{rk-izahkian} In~\cite{Izh05,Izh}, the extended
tropical semiring is also seen as the union of two copies $\R$ glued
by identifying the two $-\infty$ elements.
However, since $\N_2$ is not idempotent,
there is no possible identification of $\R$ as a subsemiring of 
$\Ti$. For instance,
the injection $\imath$ from $\rmax$ to $\Ti$, $a\mapsto (1,a)$ for $a\in\Re$
and $\ooo\mapsto \zero$, is not a morphism.
However the modulus map (Remark~\ref{extension-of-rmax})
yields a surjective morphism from $\Ti$ to $\rmax$.
\end{re}

We shall consider on $\N_2$ the identity symmetry, $-a:=a$.
Then the symmetry induced on $\Ti=\N_2\R$ as 
in Remark~\ref{extension-of-rmax}, is also the identity symmetry.
With these symmetries, we have:

\begin{pty}
$(\N_2)^\circ=\{0,2\}$, hence
$\Ti^\circ=\{\zero\}\cup(\{2\}\times (\R\setminus\{\ooo\}))$
and $\Ti^\vee=\imath(\rmax)=\{\zero\}\cup(\{1\}\times (\R\setminus\{\ooo\}))$.
\end{pty}
We shall say that an element of $\Ti$ 
is {\em real} if it
belongs to $\Ti^\vee=\imath(\rmax)$, and that it is {\em balanced}
if it belongs to $\Ti^\circ$. The same terminology applies to vectors
(meaning that every entry is real or balanced), and the notation
$\imath$ also applies to vectors or matrices (entrywise).

In $\Ti$, we have $a\balance b$ if and only if either $a,b$ have the same
modulus, or the element of $a,b$ which has the greatest modulus belongs to $\Ti^\circ$. 

In~\cite{Izh} the following notion of linear dependence over $\Ti$
is investigated.
\begin{definition} \label{DllI}
The vectors $v_1,\ldots,v_m\in \Ti ^n$ are called linearly dependent over $\Ti$ if there exist $\lambda_1,\ldots,\lambda_n\in \Ti^\vee$, not all equal to $\ooo$, such that $\lambda_1v_1\oplus\cdots \oplus \lambda_mv_m\balance \zero$. 
\end{definition}
When the vectors $v_1, \ldots,v_m$ are {\em real}, the latter relation holds if and only if when interpreting the expression $\lambda_1v_1\oplus\cdots \oplus \lambda_mv_m$ in the semiring $\rmax$, i.e., more formally, when
computing the vector $\mu_1w_1\oplus \cdots\oplus\mu_nw_n$ with $\lambda_k=\imath(\mu_k)$ and $v_k=\imath(w_k)$, the maximum is attained at least twice in every row. Hence, $w_1,\ldots,w_k$ are tropically linearly dependent in the sense
of Definition~\ref{D3}.

\subsection{Jets}
For any subsemiring $\SS$ of the semiring $(\Re_+,+,\cdot)$ of nonnegative real numbers, the semiring $\SS\R$ coincides
with the semiring of asymptotic expansions, when $p$ goes to infinity,
of the form $a e^{bp} +o(e^{bp})$ with $a\in\SS\setminus\{0\}$ and $b\in
\Re$, completed with the identically zero function, and endowed with the usual
addition and multiplication.
Taking for $\SS$ the set of nonnegative integers $\N$ and 
replacing $\R$ by the isomorphic semiring $(\Re\cup\{+\infty\},\min,+)$
(by the map $a\mapsto -a$, for the usual $-$ sign of $\Re$),
we recover the semiring introduced by
Finkelstein and Roytberg in~\cite{finkelstein} in order to
compute the number of conformations with minimum energy of 
an Ising chain at zero temperature.
Taking now $\SS=\Re_+$ and replacing $\R$ by the isomorphic 
semiring $(\Re_+,\max,\cdot)$
(by the map $a\mapsto \exp(a)$), we obtain the semiring of \NEW{jets} 
as defined by Akian, Bapat and Gaubert in~\cite{Bapat}.
There a spectral theorem on this semiring was shown which
allowed the authors to compute in some cases the
asymptotic expansion when $p$ goes to infinity 
of the Perron eigenvalue and eigenvector of a 
matrix with nonnegative entries, depending on a parameter $p$.

\section{Tropical Cramer theorems}
\label{sec-cramer}

We first recall the Cramer theorem in the symmetrized max-plus semiring, which was established by M.~Plus~\cite{Plus}. Its proof relies on an elimination argument, in which ``equations'' involving balances rather than equalities are considered. We shall see that the same elimination argument also yields a Cramer theorem in the extended tropical semiring, from which we recover a (slightly extended) version of the Cramer theorem of Richter-Gebert, Sturmfels and Theobald~\cite{RGST}.



The elimination argument uses the following properties of
the set $\smax^\vee$ of signed elements of $\smax$.
The two first ones, which are immediate, were stated in~\cite{Plus}.
The last one shows that although the balance relation is not transitive, some transitivity properties remain true when some of the data are signed.


\begin{pty}
\label{LRRR} 
For $x,y\in \smax^\vee$, we have that $x\balance y$ implies $x=y$. \qed
\end{pty}

\begin{pty}
The set $\smax^\vee\setminus \{\ooo\} =\S\setminus\smax^\circ$ is the set of all invertible elements in $\S$.  In particular, $\smax^\vee$ is stable 
with respect to product.  \qed\end{pty}

\begin{pty}[Weak transitivity of balances]\label{pty-weaktrans}
For all $b,d\in \smax$, we have 
\[
(x\in \smax^\vee, \; b\balance x \text{ and } x \balance d)\implies b\balance d \enspace .
\]
More generally, if $a\in \smax^\vee$, if $C\in \MM_{n,p}(\smax)$, 
$b\in \smax^p, d\in\smax^n$, then
\[
(x\in (\smax^\vee)^p, \; ax\balance b \text{ and } Cx \balance d)\implies Cb\balance ad \enspace .
\]
\end{pty}
\begin{proof}
Let $x\in \smax^\vee$, $b,d\in \smax$, and assume that $b\balance x$ and $x\balance d$. If $b\in \smax^\vee$, then, by Property~\ref{LRRR}, $x=b$, and so $b\balance d$. By symmetry, the same conclusion holds if $d\in \smax^\vee$. In the remaining case, we have $b,d\in \smax^\circ$ and so $b\balance d$. 

We show first the second implication when $a=\unit$. 
It follows from the construction
of $\smax$ that $x\balance b$ with $x$ signed implies that $b=x\oplus t^\circ$ for some element $t\in \smax^p$. Then, $Cb=Cx\oplus Ct^\circ \balance d$. 

Finally, if $a\in \smax^\vee$, $ax\in (\smax^\vee)^p$, 
$Cax\balance ad$ ($\smax^\circ$ is an ideal), and so, applying
the implication that we just proved to $ax\balance b$ and $Cax\balance ad$,
we deduce that $Cb\balance ad$.
\end{proof}

Since $\smax$ is a semiring with a canonical symmetry $x\mapsto \ominus x$, the determinant of a matrix is defined
by the usual formula, as in Definition~\ref{de:detS}, the term $\sgn(\sigma)$ being interpreted as $\unit$ or $\ominus \unit$ depending on the parity of $\sigma$. Similarly, the adjoint matrix $A\adj$ is defined by $A\adj=A\adjp\ominus A\adjm$ where the matrices $A\adjp$ and $A\adjm$ are defined as in Example~\ref{ex-cramer}.

The following result, which was first established in \cite{Plus}, yields
a Cramer rule for systems of balances over $\smax$.
\begin{theorem}[{Cramer theorem, \cite[Theorem 6.1]{Plus}}]\label{th-cramer}
Let $A\in {\mathcal{M}}_n(\S)$ and $b\in (\S)^n$, then
\begin{enumerate}
\item \label{cramer1} Every signed solution $x$ of the linear system
\[
Ax\balance b 
\]
satisfies the relation
$$ \detp A x \balance A\adj  b \enspace .$$
\item  \label{cramer2}
Moreover, if the vector $A\adj  b$ is signed and $\detp A$ is invertible in $\S$, then 
$$\hat x:={\detp A}^{-1} A\adj b$$
is the unique signed solution of $Ax\balance b$.
\end{enumerate}
\end{theorem}
The $i$-th entry of the vector $A\adj b$ coincides with the $i$-th Cramer
determinant, which is the determinant of the matrix obtained by replacing
the $i$th column of $A$ by the vector $b$. Hence, Theorem~\ref{th-cramer} gives
an analogue of Cramer rule.

We next prove Theorem~\ref{th-cramer}, along the lines of~\cite{Plus},
in a way which will allow us to derive a similar Cramer
theorem over the extended tropical semiring.
\begin{proof}[Proof of Theorem~\ref{th-cramer}]
We first establish Assertion~\eqref{cramer1} under the assumption that $\detp A$ is signed,
by induction on the dimension $n$. When $n=1$, the result is obvious. By 
expanding $\detp A$ with respect to the $k$-th column of $A=(a_{ij})$, we get 
\[
\detp A= \bigoplus_{l}(\ominus \unit)^{l+k} a_{lk}\detp{A(l|k)} \enspace 
\]
(recall that $A(l|k)$ denotes the submatrix of $A$ in which row $l$ and
column $k$ are suppressed).
Since $\detp A$ is signed, for any $k$ there must exist at least one $l$ (depending on $k$) such that $\detp{A(l|k)}$ is signed. Possibly after permuting the rows and the columns of $A$, we may assume that $l=k=n$, and we set $A':=A(n|n)$. The system $Ax\balance b$ can be rewritten as 
\begin{align}
A'x' &\balance b' \ominus c x_n \label{e-sys-1}\\
d x'&\balance b_n  \ominus a_{nn} x_n \label{e-sys-2}
\end{align}
where $b', c$ and $d$ denote the two column vectors and the row vector
of dimension $n-1$ such that $b'_i=b_i$, $c_i=A_{in}$
and $d_i=A_{ni}$ for $i=1,\ldots,n-1$.  By applying the induction assumption
to~\eqref{e-sys-1},
we get
\begin{align}\label{e-sys-3}
\detp{A'}x' \balance (A')\adj (b'\ominus cx_n) \enspace .
\end{align}
Using the weak transitivity property to~\eqref{e-sys-2} and~\eqref{e-sys-3}, we deduce that 
\[
d(A')\adj (b'\ominus cx_n) \balance \detp{A'}(b_n \ominus a_{nn}x_n) \enspace ,
\]
that is
\[
(\detp{A'}a_{nn} \ominus d(A')\adj c)x_n \balance \detp {A'}b_n 
\ominus d(A')\adj b' \enspace .
\]
In the factor at the left hand side, we recognize the expansion
of $\detp A$, whereas at the right hand side, we recognize the expansion
of the $n$-th Cramer
determinant. Hence, $\detp A x_n\balance  (A\adj b)_n$. Since the choice of the column was arbitrary in the previous argument,
it follows that $\detp A x\balance A\adj b$, which concludes the induction.

It remains to consider the case where $\detp A$ is balanced. It suffices to check that $\detp A x_j\balance (A\adj b)_j$ for every index $j$ such that 
$(A\adj b)_j$ is signed (for the other indices, the balance relation trivially holds).
We assume without loss of generality that $j=n$,
and we consider the system:\[
\left[ 
\begin{array}{cccc}A_{1}&\ldots&A_{n-1}&b
\end{array}
\right]
\left[ 
\begin{array}{cc}x_{1}\\ \ldots\\ x_{n-1}\\ \ominus \unit
\end{array}
\right]
\bal \ominus A_{n}x_{n}\enspace ,
\]
where $A_j$ denotes the column $j$ of $A$. This system
can be written as $A''x''\balance b''$ with 
$A'':=[A_{1},\ldots,A_{n-1},b]$, $x'':=[x_1,\ldots,x_{n-1},\ominus \unit]^t$
and $b'':=A_nx_n$.
Observe that the determinant of $A''$ is precisely $(A\adj b)_n$,
which we assumed to be signed.
Hence, we apply Assertion~\eqref{cramer1}
of the theorem, which is already proved for matrices
with a signed determinant, to the system $A''x''\bal b''$, which
gives:
\[
(A\adj b)_{n}\left[
\begin{array}{cc}x_{1}\\ \ldots\\ x_{n-1}\\ \ominus \unit
\end{array}
\right] \balance \ominus \left[
\begin{array}{cccc}A_{1}&\ldots&A_{n-1}&b
\end{array}
\right]\adj A_{n}x_{n} \enspace .
\] 
Taking the last entry, we get
\[
(A\adj b)_{n}\bal 
\detp A x_{n}
\]
which completes the proof of Assertion~\eqref{cramer1}.

To prove Assertion~\eqref{cramer2}, 
we deduce from the Cramer identity~\ref{ex-cramer} that
\[
AA\adj \balance \detp{A}I \enspace .
\]
If $\detp{A}$ is invertible, right multiplying this balance relation by $\detp{A}^{-1}b$, we get that $\hat x$ satisfies $A\hat x \balance b$. Conversely, by Assertion~\eqref{cramer1},
every signed solution $x$ of $Ax\balance b$ satisfies
$\detp{A}x\balance A\adj b$. Since $\detp{A}$ is invertible, the latter
condition is equivalent to $x\balance \hat{x}$. We deduce
from Property~\ref{LRRR} that $x=\hat{x}$, which completes the proof of Assertion~\eqref{cramer2}.
\end{proof}
As a corollary of this theorem, we recover a Cramer rule for two sided linear systems over $\R$, rather than for balances over $\smax$.
\begin{corollary}
Let $A',A''\in \M(\R)$ and $b',b''\in (\R)^n$. Then, any solution
$x\in \R^n$ of the system $A'x\oplus b'= A''x\oplus b''$ satisfies 
$|A|x\balance A\adj b$ where $A:=A'\ominus A''$ and $b:=b''\ominus b'$.
In particular, if $|A|$ is invertible, and if the vector
$A\adj b$ is signed, the solution is necessarily unique, and the solution
exists if and only if the entries of $|A|^{-1}A\adj b$ are sign-positive.
\end{corollary}
\begin{proof}
If $A'x\oplus b'= A''x\oplus b''$, then $Ax\balance b$, and so, by Theorem~\ref{th-cramer}, $|A|x\balance A\adj b$. Hence, the uniqueness result follows from the previous theorem. If in addition,  the entries of $x:=|A|^{-1}A\adj b$ are sign-positive, $x$ may be thought of as an element of $\R^n$ rather than $\smax^n$, and
we have $Ax\balance b$, which implies that $A'x\oplus b'\balance A''x\oplus b''$.
It follows from Property~\ref{LRRR} that $A'x\oplus b'=A''x\oplus b''$. 
\end{proof}
Let us now replace the symmetrized max-plus semiring $\smax$ by the extended
max-plus semiring $\Ti$. Recall that an element of $\Ti$ is real
if it belongs to $\Ti^\vee=\imath(\rmax)$ and balanced if it belongs
to $\Ti^\circ$ (so, the real elements of $\Ti$ play the role
of the signed elements of $\smax$).

The properties~\ref{LRRR}--\ref{pty-weaktrans} are easily seen
to hold when $\smax$ is replaced by $\Ti$. Besides determinantal identities
(expansions of determinants and Cramer identities~\ref{ex-cramer}),
these properties are the only ingredient of the proof of Theorem~\ref{th-cramer}, and so,
the analogous result in $\Ti$ is true.
\begin{theorem}[{Cramer theorem in the extended tropical semiring}]\label{th-cramerI}
Let $A\in {\mathcal{M}}_n(\Ti)$ and $b\in {\Ti}^n$, then 
\begin{enumerate}
\item Every real solution $x$ of the linear system
\[
Ax\balance b \]
satisfies the relation
$$ \detp A x \balance A\adj  b \enspace .$$
\item Moreover, if the vector $A\adj  b$ is real and $\detp A$ is invertible in $\Ti$, then 
$$\hat x:={\detp A}^{-1} A\adj b$$
is the unique real solution of $Ax\balance b$.
\qed 
\end{enumerate}
\end{theorem}
Since the symmetry of $\Ti$ is the identity map, the determinant $\detp{A}$ of a matrix $A\in \MM_n(\Ti)$ coincides with the permanent $\per(A)$.

As a corollary, we next derive a Cramer rule for the systems of tropical linear equations already considered by Richter-Gebert, Sturmfels and Theobald~\cite{RGST}. The following notion used in~\cite{RGST} was introduced by Butkovi\v{c} under the name of \NEW{strong regularity}, see for instance~\cite{butkovip94,But}. To avoid the risk of confusion
with the notion of Von Neumann regularity
(matrices with a generalized inverse), we shall keep the terminology
of~\cite{RGST}.
\begin{definition}
A matrix $A\in \MM_n(\rmax)$ is said to be \NEW{tropically singular} if
the maximum is attained twice in the expression~\eqref{E1} of the permanent
of $A$, i.e., with the usual notation, in
\[
\per A = \max_{\sigma\in \mathfrak{S}_n}(a_{1\sigma(1)}+\cdots+ a_{n\sigma(n)}) \enspace .
\]
\end{definition}
Note that if $\per A=\zero$, $A$ is tropically singular. A matrix
$A\in \MM_n(\rmax)$ is tropically singular if and only if $\detp{A}$
when interpreted in $\Ti$ is balanced, meaning that
$\detp{\imath(A)}\in \Ti^\circ$. So $\detp{\imath(A)}$ is invertible
in $\Ti$ if and only if $A$ is tropically nonsingular, which
provides a further justification for the name of the notion.

\begin{re}
In an arbitrary semiring, we may define tropically singular matrices by requiring that for some subset ${\cal T}$ of $\allperm_n$ distinct from the empty set and  from $\allperm_n$, 
$$\sum\limits_{\sigma\in {\cal T}} a_{1\sigma(1)}\cdots a_{n\sigma(n)} = \sum\limits_{\sigma\in \allperm_n\setminus {\cal T}} a_{1\sigma(1)}\cdots a_{n\sigma(n)}.$$ 
\end{re}
In the next corollary, we denote by $B_i$ the $i$th Cramer matrix of $(A,b)$,
obtained by replacing the $i$th column of $A$ by $b$.
The $i$th Cramer permanent is defined as $\per B_i$. 
This corollary is a mere specialization of Theorem~\ref{th-cramerI}
to matrices and vectors with real entries.
\begin{corollary}[Cramer theorem for tropical linear equations, compare with {\cite[Corollary~5.4]{RGST}}]\label{coro-sturmf}
Let $A=(a_{ij})\in \MM_n(\rmax)$, $b,x\in \rmax^n$. 
Assume that for every row index $1\leq i\leq n$, the maximum in the expression
\begin{align}
\label{e-tropicaleq}
\bigoplus_{j}a_{ij}x_j \oplus b_i 
\end{align}
is attained at least twice. Then, for all $1\leq i\leq n$, if
we expand $\per B_i$ and $\per A$ in
\[
(\per A) \, x_i \oplus  \per B_i  \enspace,
\]
the maximum is attained at least twice in the global expression.
Moreover, if $A$ is tropically nonsingular and if every Cramer matrix $B_i$
is tropically nonsingular or has a zero permanent, then
$\hat x:= ((\per A)^{-1}\per B_i)_{1\leq i\leq n}$ is the unique vector $x\in\rmax^n$ such that the maximum in Expression~\eqref{e-tropicaleq} is attained at least twice, for every $1\leq i\leq n$.\qed
\end{corollary}
A result closely related to the second part of this corollary is proved 
by Richter-Gebert, Sturmfels, and Theobald in~\cite[Coro.~5.4]{RGST},
by an elegant technique building an an earlier study
of special transportation polytopes by Sturmfels and Zelevinsky~\cite{SZ}.
It is shown in~\cite{RGST}, under the same assumptions,
that the tropical Cramer permanents are given (up to an additive constant)
by the unique optimal solution of a transport problem, and that the dual
variables of this transportation problem are unique.


%

The following theorem shows that the existence part of Theorem~\ref{th-cramer} does not require the condition that all the Cramer determinants be signed. 
This theorem was proved in~\cite{Plus} when the determinant of $A$ is non-zero.
An extension of this proof to the general case appeared in~\cite{gaubert92a};
a more recent presentation can be found in~\cite{AGG2}.
The proof arguments rely of the convergence of 
an iterative Jacobi-type algorithm, introduced in~\cite{Plus},
which allows one to solve the system $Ax\balance b$.
\begin{theorem}[{\cite[Th.~6.2]{Plus}}]\label{th-jac}
Let $A\in \M(\smax)$, and assume that $\detp A\neq \zero$ (but possibly $\detp{A}\balance \zero$). Then, for every $b\in \smax^n$, there exists a signed
solution $x$ of $Ax\bal b$. 
\end{theorem}

A result analogous to the previous one, but with the extended
tropical semiring $\Ti$ instead of $\smax$, is proved in~\cite{AGG2}.

Theorem~\ref{th-cramer} has an homogeneous analogue, which was stated in~\cite{Plus} and proved in~\cite[Ch.~3,S.~9]{gaubert92a}, see also~\cite{ButGa}.
\begin{theorem}[{\cite[6.5]{Plus}}]
\label{TGMa}
Let $A \in\M(\smax)$. Then there exists $x\in (\smax^\vee)^n\setminus \{\zero\}$ such that $Ax\balance \zero$ if and only if $\detp A \balance \ooo$.
\end{theorem}
The ``only if'' part is obtained by taking $b=\zero$ in the second
part of Theorem~\ref{th-cramer}. 
The ``if'' part is proved in~\cite{gaubert92a} by combining Theorem~\ref{th-jac}
with an idea of Gondran and Minoux~\cite{GM2}. Indeed,
the special case in which $A\in \MM_n(\rmax)$ can be stated
as follows.
\begin{corollary}[{Gondran-Minoux theorem~\cite{GM2}}]
\label{TGM}
Let $A \in\M(\rmax)$. Then $\detp A \balance \ooo$ if and only if columns of $A$ are linearly dependent in the sense of Gondran and Minoux (Definition~\ref{D2}).
\end{corollary}
%
As a corollary of Gondran-Minoux theorem, we obtain
an analogue of the famous Radon theorem in convex geometry, which
shows that $n+1$ vectors in dimension $n$ can be partitioned in two
subsets in such a way that the two convex cones generated by these
subsets have an intersection that is not reduced to the origin.
The max-plus Radon theorem was
first derived from the Gondran-Minoux theorem by P.~Butkovi\v{c}~\cite[Theorem 4.7]{But} in the special case of vectors without a $-\infty$ entry. The latter
restriction turns out to be unnecessary, since a more general derivation, combining the Gondran-Minoux theorem and the Cramer
theorem, was sketched in~\cite{AG}, we detail the argument below for the sake
of completeness. Briec and Horvath gave a different proof,
by seeing tropical convex sets as limits of classical convex sets~\cite{BriecHorvath04}. The Radon theorem is also proved by Gaubert and Meunier in~\cite{meunier}, where max-plus analogues
of other results of discrete convex geometry are established.

\begin{corollary}[Max-plus Radon theorem, see~\cite{But}, \cite{BriecHorvath04}, {\cite[{p.~13}]{AG}}, \cite{meunier}]
\label{LS}
Every family of $n+1$ vectors of $\R ^n$ is linearly dependent in the sense of Gondran and Minoux.
\end{corollary}
\d
Let $v_1,\ldots, v_{n+1}$ denote vectors of $\R^n$, and let $V_i$ denote the matrix constructed by concatenating all these (column) vectors but the $i$th.

Assume first that there exists $i$ such that $\detp {V_i}\balance \ooo$. Then,
by Corollary~\ref{TGM}, the columns of $V_i$ are linearly dependent in the sense of Gondran and Minoux. A fortiori, $v_1,\ldots,v_{n+1}$ are linearly dependent
in this sense.

Assume now that all the determinants  $\detp {V_i}$ are unbalanced ($i=1,\ldots,n+1$). Then by the Cramer rule (Theorem~\ref{th-cramer}), the system
$V_{n+1}x\balance v_{n+1}$ admits a (unique) non-zero signed solution $x$,
and so the vectors  $v_1,\ldots,v_{n+1}$ are linearly dependent in the sense of Gondran and Minoux.
\dd

The Cramer theorems~\ref{th-cramer} and~\ref{th-cramerI} raise the issue
of computing determinants or permanents
in the semirings $\rmax$, $\smax$ or $\Ti$. First, we observe
that if $A\in \MM_n(\rmax)$, computing $\per A$ is nothing
but the classical optimal assignment problem, which can be solved in polynomial
time. Hence, all the Cramer permanents of $(A,b)$ (for some $b\in\rmax^n$)
together with $\per A$ 
could be obtained by solving $n+1$ assignment problems.
Alternatively, the method of Richter-Gebert, Sturmfels, and Theobald~\cite{RGST}
shows
that one can compute at once all the Cramer permanents together
with the permanent of $A$, up to a common additive constant, by solving
a single network flow problem. 
The Jacobi algorithm of M.~Plus~\cite{Plus} leads to
a third method. 
In~\cite{AGG2}, the latter method is further discussed and compared
with the one of~\cite{RGST}.

The compution of determinants over $\Ti$ or $\smax$ reduces
to a purely combinatorial problem, thanks to the following
technique. 
Let $A\in \MM_n(\rmax)$. By applying a standard assignment algorithm,
like the Hungarian algorithm, as soon as $\per A\neq\zero$,
we get optimal dual variables $u_i,v_j\in \mathbb{R}$, for $1\leq i,j\leq n$, which are such that
\[
A_{ij}\leq u_i + v_j \enspace ,\qquad \text{and}\qquad 
\per A=\sum_i u_i +\sum_j v_j \enspace .
\]
By the complementary slackness property, the optimal permutations $\sigma$
are characterized by the condition that $A_{i\sigma(i)}=u_i+v_{\sigma(i)}$.
After multiplying $A$ by a permutation matrix,
we may always assume that the identity is a solution of the
optimal assignment problem. 
Then, we define the digraph $G$ with nodes $1,\ldots,n$, and an arc from
$i$ to $j$ whenever $A_{ij}= u_i + v_j$.
Butkovi\v{c} proved two results which can be formulated
equivalently as follows.
\begin{theorem}[See~{\cite{butkovip94} and~\cite{butkovic95} }]\label{thbut}
Let $A\in \MM_n(\rmax)$, and assume that $\per A\neq\zero$.
Then, checking whether the 
optimal assignment problem has at least two optimal solutions reduces
to finding a cycle in the digraph $G$, whereas checking whether
it has at least two optimal solutions of a different parity reduces
to finding an (elementary) even cycle in $G$.
\end{theorem}
The existence of a cycle can be checked in linear
time (e.g.~by a depth first search algorithm).
The polynomial time character of the even cycle problem
is a deep result of Robertson, Seymour, and Thomas~\cite{RST99}.

If $A\in \MM_n(\rmax)$, one can readily design from the first part 
of Theorem~\ref{thbut} a polynomial time
algorithm to compute the determinant of the matrix $\imath(A)$,
i.e., the determinant of $A$ interpreted in the semiring $\Ti$.
The second part of this theorem also leads
to a polynomial time algorithm to compute the determinant of $A$,
interpreted in the semiring $\smax$.
The determinant of a matrix in $\MM_n(\Ti)$ can be computed
in polynomial time along the same lines. More generally,
as is detailed in~\cite{AGG2},
computing the determinant of a matrix
in $\MM_n(\smax)$ reduces to checking whether all the terms
of the expansion of the determinant of a matrix with entries
in $\{\pm 1,0\}$ have the same sign (here, the determinant
is evaluated in the usual algebra). This problem also reduces
to the even cycle problem. It has been considered
within the theory of ``sign solvable systems''~\cite{shader}.
The latter deals with those linear systems having solutions the sign of which
is uniquely determined by the sign of the coefficients. 
We refer the reader to~\cite{ButGa} for a further discussion
of the relation between the symmetrization of the max-plus semiring
and the sign solvability theory.


\section{Rank functions}\label{S2}

In this section we review several notions of rank for matrices over semirings.
Different points of views, which yield equivalent definitions in the case of fields, lead to different notions in the case of semirings.
Indeed, we may define the rank in terms of matrix factorization,
in terms of determinant, or in terms of independence of the rows
or columns.

\begin{definition} \label{Df} Let $\SS$ be any semiring.
The \NEW{factor rank} $\facrk(A)$ of a matrix $A\in \MM_{mn}(\SS)$ is
the smallest integer $k$ such that $A=BC$ for some matrices $B\in \MM_{mk}(\SS)$ and $C\in \MM_{kn}(\SS)$.
\end{definition}
By convention, a matrix with zero coefficients has factor rank $0$.

Note that the factor rank of $A$ is equal to the minimum
number of matrices of factor rank 1 the sum of which is equal to~$A$. Also for any submatrix $A'$ of $A$ we have $\facrk(A')\le \facrk(A)$, see~\cite{BP4}. See also \cite{CR} for more details. 
The name Schein rank has also been used for the factor rank,
particularly in the case of Boolean matrices~\cite{kim82}. For matrices over the max-plus algebra, the factor rank is also known as the Barvinok rank, since
it appeared in a work of Barvinok, Johnson and Woeginger on the MAXTSP problem~\cite{barvinok}. 

\begin{definition}
\NEW{The tropical rank} of $A\in \MM_{mn}(\SS)$, denoted by $\trop(A)$, is the biggest integer $k$ such that $A$ has a tropically non-singular $k\times k$-submatrix.
\end{definition}
\begin{definition}
The \NEW{determinantal rank} of $A\in\MM_{mn}(\SS)$, denoted by $\rkdet (A)$, is the biggest integer $k$ such that there exists a $k\times k$-submatrix $A'$ of $A$ with $| A'| ^+\ne | A'| ^-$.
\end{definition}
Observe that a matrix $A\in \MM_n(\rmax)$ has tropical rank $n$
if and only if $\detp{A}$, when interpreted in the extended tropical
semiring $\Ti$ (i.e., $\detp{\imath(A)}$) is an invertible element of $\Ti$,
whereas $A$ has determinantal rank $n$ if and only if $\detp{A}$, when evaluated
in the symmetrized tropical semiring $\smax$, is an invertible element of $\smax$. 
\begin{re}\label{rk-tropdet}
It follows readily from the definition that $ \trop(A)\le  \rk_{\det}(A)$
for all $A\in\Mnm(\R)$.
\end{re}
The following rank notion is usually considered in combinatorics.
\begin{definition}
The \NEW{term rank} of a matrix $A\in\Mnm(\SS)$, denoted $\termrk(A)$
is defined as the minimal number of lines (rows and columns) necessary to cover all the non-zero elements of~$A$, or equivalently (by K\"onig theorem) as the maximal number of non-zero entries of $A$ no two of which lie on the same row or column.
\end{definition}
\begin{re}\label{rk-sasha}
It is proved in~\cite[Prop.~3.1]{BG} that the inequality 
\[
\facrk(A)\leq \termrk(A)
\]
holds for matrices with entries in an arbitrary semiring.
\end{re}
\bigskip
We now turn to the definitions of matrix rank, based on the different notions of linear  independence, introduced in Section~\ref{S1}.

\begin{definition} \label{Dmr}
The \NEW{maximal row rank} of a matrix $A\in\Mnm(\SS)$
in the weak, Gondran-Minoux, or tropical sense (see Definitions~\ref{D1}, \ref{D2}, and~\ref{D3}), denoted respectively by $\mrw(A)$, $\mrgm(A)$, and $\mrtrop(A)$, is the maximal number $k$ such that $A$ contains $k$ weakly, Gondran-Minoux, or tropically, linearly independent rows, respectively.
\end{definition}
\begin{re}\label{rk-obviousineq}
Due to the implications between the different independence notions, we 
readily get $\mrtrop(A)\leq \mrgm(A)\leq \mrw(A)$.
\end{re}
\begin{re}[\bf Monotonicity of rank functions]\label{monrank}
Let $A\in \Mnm(\SS)$ and let $B$ be a submatrix of $A$. Then
\begin{enumerate}
\item $\facrk(B)\leq \facrk(A)$;
\item $\trop(B)\le \trop(A)$;
\item $\rk_{{\det}} (B)\le \rk_{{\det}}(A)$;
\item $\mr_{\star}(B)\le  \mr_{\star}(A)$ for $\star\in \{\mathrm{w},\mathrm{GM},\mathrm{t}\}$. 
\end{enumerate}
Indeed, the first three inequalities follow immediately from the definitions.
For the last one, say in the case of the Gondran-Minoux
independence, we note that by the definition, if some
vectors constitute a linearly
dependent family, then so do their restriction
to an arbitrary set of coordinates. Thus, every family of linearly independent
rows of $B$ yields a family of linearly independent rows of $A$.
Since $ \mrgm$ is the maximal number of vectors of such families,
the result follows.
\end{re}

\begin{definition} \label{Dr}
 The \NEW{row rank} of a matrix $A\in \MM_{m,n}(\SS)$, denoted by $\ror(A)$, is the 
weak dimension $\wdim(\rosp(A))$ 
of the linear span $\rosp(A)$ of the rows of $A$.
\end{definition}

\begin{re}
It is proved in \cite{BP4} that $\ror(B)\le \ror(A)$ if $B$ is obtained by deleting some columns of $A$. However since the weak dimension is not in general increasing (see Remark~\ref{re:weaknonincr}), we may have $\ror(C)>\ror(A)$,
for matrices $C$ obtained by deleting some rows of $A$,
as is shown in Example~\ref{Xi} below.
\end{re}

%
\begin{ex} \label{Xi}
Consider the matrix
$$
 Y = \left[ \begin{array}{ccc}
\ooi& \ooo&  \ooo\\
 \ooo&\ooi&\ooo\\
\ooo&\ooo&\ooi\\
\ooi&\ooo&\ooi\\
\ooi&\ooi&\ooo\end{array}\right]\in \MM_{5,3}(\R)$$
and its proper submatrix
$$
X = \left[ \begin{array}{ccc}
 \ooo&\ooi&\ooo\\
\ooo&\ooo&\ooi\\
\ooi&\ooo&\ooi\\
\ooi&\ooi&\ooo\end{array}\right] \in \MM_{4,3}(\R).
$$
Since $\rosp(Y)=\R^3$ and the rows of $X$ are weakly independent,
we see that $\ror(Y)=3<4=\ror(X)$.
\end{ex}

\begin{lemma} \label{C:ext_rays}
The row rank of a matrix $A\in \Mnm(\R)$ is equal to the number of extremal rays of the row space $\rosp(A)$ of $A$.
\end{lemma}
\d This follows from Corollary~\ref{C1:ext_rays}.
\dd

In the theory of general semirings, the following rank function is considered:

\begin{definition} The \NEW{spanning row rank} of 
a matrix $A\in \Mnm (\SS)$, denoted $\sr(A)$, is the minimal number of rows of $A$ which generate over $\SS$ the row space of $A$.
\end{definition}

\begin{re}\label{rk-general}
Note that for matrices over general semirings, we have that $\ror(A)\le \sr(A)\le \mrw(A)$. There are semirings such that there exist matrices $A$ for which $\ror(A)<\sr(A)<\mrw(A)$, see~\cite{BG}.
However over max-plus algebra the situation with the first two functions is different.
\end{re}

\begin{theorem} \label{TCGB}
The identity $\ror(A)=\sr(A)$ holds for all $A\in \MM_{n,m}(\rmax)$.
\end{theorem}
\d By definition, every extremal rays of  $\rosp(A)$ is generated by
 one row of $A$, hence from Corollary~\ref{C1:ext_rays}, there exists a subset of the set of rows of $A$ which is a weakly independent generating family of $\rosp(A)$. This shows that $\sr(A)\leq \ror(A)$, and since the other inequality is always true, we get the equality.
\dd

\begin{ex} \label{ex:mr_w}
For any positive integer $n$ let us consider the matrix $$A=\left[ \begin{array}{ccccccc} 
\ooi & \ooo  & \ooo & x_1 & x_2 &\ldots & x_n\\
\ooo &\ooi & \ooo & \ooi & \ooi &\ldots & \ooi\\
\ooo  & \ooo & \ooi & -x_1 & -x_2 &\ldots & -x_n\end{array} \right]^t\in {\cal M}_{n+3,3}(\R),$$
where $x_1,\ldots ,x_n \in \R$ are pairwise different and different from $\ooi,\ooo$.
Then it is easy to see that $\ror(A)=3$, but $\mrw(A)=n$, cf. Example~\ref{ExD1D2}.
\end{ex}

This example shows that the rank $\mrw$ has somehow a pathological behavior.

\begin{definition} \label{Der}
For a matrix $A\in \MM_{m,n}(\SS)$, we define the \NEW{enveloping row rank}
of $A$ with respect to any linear dependence notion as the corresponding
enveloping dimension (see Definition~\ref{defi-envdim})
of the subset of $\SS^n$ obtained from the rows of $A$, or equivalently 
of the subsemimodule $\rosp(A)$ of $\SS^n$.
We denote respectively by $\erw(A)$, $\ergm(A)$, and $\ertrop(A)$,
the enveloping row rank with respect to the weak, Gondran-Minoux and tropical linear dependence.
\end{definition}

\begin{re}\label{re-env-rank}
 From Remark~\ref{re-env}, we get that
$\erw(A)\leq \ergm(A)\leq \ertrop(A)\leq n$, and that
$\erw(A)\leq m$.
\end{re}
The following elementary observation shows
that the enveloping rank with respect to the weak linear
dependence notion is nothing but the factor rank.
\begin{prop} \label{P_er=f}
 Let $A\in \Mnm(\SS)$. Then
$\erw(A) = \facrk(A)$.
\end{prop}
\d Let us check that $\erw(A)\le \facrk(A) $. We set $f:=\facrk(A)$. Then $A=BC$ for some $B\in {\cal M}_{m\,f},$ $C\in {\cal M}_{f\,n}$. This means  that the rows of $A$ are elements of the row space of $C$, hence
$\erw(A)=\edimw(\rosp(A))\leq \edimw(\rosp(C))$ which is 
equal to the enveloping dimension of the set of rows of $C$, and
since the cardinality of this set is less or equal to $f$, we get that
$\erw(A)\leq f$ by Remark~\ref{re-env-rank}.

Let us show the opposite inequality. Let $r:=\erw(A)$. By definition, there exist row vectors ${\bf v}_1,\ldots, {\bf v}_r\in \SS^n$ generating the rows of $A$. Hence there exist elements $\alpha_{i,j}\in\SS$, $i=1,\ldots, m,$ $j=1,\ldots, r$, such that the $i$th row of $A$ is equal to $\sum\limits_{j=1}^r \alpha_{i,j}{\bf v}_j $. 
Thus $A=BC$, where $B=[\alpha_{i,j}]$ an $C$ is the matrix with rows ${\bf v}_1$, \ldots, ${\bf v}_r$.
\dd

\bigskip

The corresponding ranks can be defined by considering columns
instead of rows.
For instance, the column rank $\cor(A)$ of $A$ is the weak dimension of the linear span $\cosp(A)$ of the columns of $A$, i.e., $\cor(A)=\ror(A^t)$. Similarly, $\mctrop(A):=\mrtrop(A^t$), $\ecw(A)=\edimw(\cosp(A))=\erw(A^t)$, etc. 
The ranks with respect to rows and columns may differ.
In particular, the matrix $X$ from Example~\ref{Xi} is such that $\cor(X)\ne \ror(X)$ and Proposition~\ref{mrmcGM} below shows that we may have $\mrgm(X)\neq \mcgm(X)$.

\bigskip

\begin{re}\typeout{TODO}
When $\SS=\rmax$, Develin, Santos, and Sturmfels~\cite{DSS} considered
an additional rank, the Kapranov rank, which is defined by
thinking of max-plus (or rather min-plus) scalars as images of Puiseux series by
a non-archimedean valuation. This notion is
of a different nature, and therefore is out of the scope of this paper.
\end{re}



\section{Comparison of rank functions}
We now give the main comparison results for rank functions over $\rmax$.
The proof of these rely on the results on
max-plus linear systems and in particular on the ``Cramer rules''
established in Section~\ref{sec-cramer}.

%
%
%

\begin{lemma} \label{LO1new}
For any $A\in\Mnm(\R)$, we have
$\rkdet(A)\leq \mrgm(A)$
and $\trop (A)\leq \mrtrop(A)$.
\end{lemma} 
\begin{proof}
Let $k:=\rkdet(A)$, and let $A'$ denote a $k\times k$ submatrix of $A$
such that $\detp{A}$ is invertible in $\smax$. Then, it follows from
the Cramer theorem~\ref{th-cramer} that there is no signed row vector
$x\neq\zero$ such that $xA'\balance \zero$. Hence, the rows of $A'$, and a fortiori the corresponding rows of $A$, are linearly
independent in the sense of Gondran and Minoux. This shows that
$\rkdet(A)\leq \mrgm(A)$. A similar argument with
$\Ti$ instead of $\smax$, and Theorem~\ref{th-cramerI} instead
of Theorem~\ref{th-cramer}, shows that $\trop(A)\leq \mrtrop(A)$.
\end{proof}

The second inequality in Lemma~\ref{LO1new} also follows
from a result of Izhakian~\cite[Th.~3.4]{Izh}. 
Moreover, Izhakian proved the following theorem.
\begin{theorem}[{\cite[Th.~3.6]{Izh}}]\label{th-add-izh}
If $A\in\M(\rmax)$, then $\trop(A)=n$ if and only if $\mrtrop(A)=n$.
\end{theorem}

The following analogous result in $\S$ 
is an immediate consequence of Theorem~\ref{TGMa}.
\begin{theorem}\label{th-gmnew}
If $A\in\M(\rmax)$, then $\rk_{\det}(A)=n$ if and only if $\mrgm(A)=n$.
\qed
\end{theorem}

We shall see in Proposition~\ref{mrmcGM} below that in general
$\rk_{\det}(A)$ and $\mrgm(A)$ may differ. However,
Theorem~\ref{Th-trop-rect} below, stated by Izhakian in~\cite{Izh},
shows that the analogous rank notions coincide 
when the symmetrized max-plus semiring $\S$ 
is replaced by the extended tropical semiring $\Ti$.

Theorem~\ref{Th-trop-rect} may seem quite surprising.
Indeed, the proof
of Theorem~\ref{th-gmnew} via Theorem~\ref{TGMa}, 
as well as Izhakian's proof of Theorem~\ref{th-add-izh}, 
can be interpreted in terms of network flows arguments. 
The extension of such flow arguments to general rectangular matrices
seems foredoomed, because 
Sturmfels and Zelevinsky showed in~\cite{SZ} that
the Newton polytope of the product of the maximal minors of a general
rectangular matrix is not a transportation polytope, unless
the numbers of rows and columns differ of at most one unit.
Hence, different techniques must be used.
In~\cite{Izh}, Izhakian gives elements of proof of Theorem~\ref{Th-trop-rect} 
relying on a reduction to the square case, by an inductive argument. 
We believe that this proof strategy can lead to the result, 
however, further arguments are needed. 
In~\cite{AGG}, we prove directly the result
in the rectangular case, using a different approach
in which linear independence is expressed
in terms of a zero-sum mean payoff game problem. We also show
that the rectangular case can be derived from the
square case by applying the tropical analogue
of Helly's theorem~\cite{BriecHorvath04,gauser,meunier}.



\begin{theorem}[{See \cite{Izh}, \cite{AGG}}] \label{Th-trop-rect}
For any $A\in\Mnm(\rmax)$, we have
$\trop (A)=\mrtrop(A)=\mctrop(A)$.
\end{theorem} 

The following elementary result completes the comparison between
the various rank functions.
\begin{lemma}\label{lem-ineqhalf}
For any $A\in\Mnm(\R)$, we have 
$\mrgm(A) \le \facrk(A) \le \ror(A). $
\end{lemma}
\d
We prove the first inequality.
Let $r:=\facrk(A)$. If $r=m$ then we are done. So let us assume that $r<m$. We have to check that any $r+1$ rows of $A$ are Gondran-Minoux linearly dependent. Up to a permutation we may consider the first $r+1$ rows: $A_{1\circ},\ldots, A_{r+1\circ}$. 
(Here, $F_{i\circ}$ denotes the $i$th row of $A$.)
By Definition~\ref{Df} there exist matrices $B\in{\cal M}_{m\,r}$, $C\in {\cal M}_{r\, n}$ such that $A=BC$. By Corollary~\ref{LS} the first $r+1$ rows of $B$ are Gondran-Minoux linearly dependent. Thus there exist subsets $I,J\subset K:=\{1,\ldots, r+1\}$, $I\cap J=\emptyset$, $I\cup J =K$ and scalars $\lambda_1,\ldots, \lambda_{r+1}$ not all equal to $\ooo$ such that
$$ \bigoplus\limits_{i\in I}\lambda_i B_{i\circ}= \bigoplus\limits_{j\in J}\lambda_j B_{j\circ}.$$ 
Right multiplying this equality by $C$, we deduce that the same equality
holds for the rows of $A$, and so $\mrgm(A)\le r$.

The second inequality is proved in~\cite{BG}.
\dd



Gathering Remarks~\ref{rk-sasha},~\ref{rk-tropdet} and~\ref{rk-general}, the first part of Lemma~\ref{LO1new}, and Lemma~\ref{lem-ineqhalf}, 
we obtain the following comparison theorem.
\begin{theorem} \label{TEq}
For $A\in\Mnm(\R)$ the ranks of $A$ are ordered as indicated in the Hasse diagram of Figure~\ref{fig1} (when two ranks are connected by a segment, the rank
at the top of the segment is the bigger one). 
\end{theorem}

\begin{figure}[htbp]
\begin{picture}(0,0)%
\includegraphics{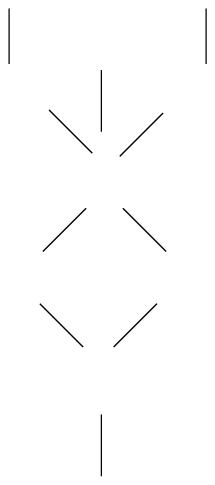}%
\end{picture}%
\setlength{\unitlength}{2590sp}%
\begingroup\makeatletter\ifx\SetFigFontNFSS\undefined%
\gdef\SetFigFontNFSS#1#2#3#4#5{%
  \reset@font\fontsize{#1}{#2pt}%
  \fontfamily{#3}\fontseries{#4}\fontshape{#5}%
  \selectfont}%
\fi\endgroup%
\begin{picture}(1827,4107)(3676,-6637)
\put(4321,-5686){\makebox(0,0)[lb]{\smash{{\SetFigFontNFSS{10}{12.0}{\familydefault}{\mddefault}{\updefault}{\color[rgb]{0,0,0}$\rkdet(A)$}%
}}}}
\put(5086,-4876){\makebox(0,0)[lb]{\smash{{\SetFigFontNFSS{10}{12.0}{\familydefault}{\mddefault}{\updefault}{\color[rgb]{0,0,0}$\mrgm(A)$}%
}}}}
\put(3691,-4876){\makebox(0,0)[lb]{\smash{{\SetFigFontNFSS{10}{12.0}{\familydefault}{\mddefault}{\updefault}{\color[rgb]{0,0,0}$\mcgm(A)$}%
}}}}
\put(3691,-2761){\makebox(0,0)[lb]{\smash{{\SetFigFontNFSS{10}{12.0}{\familydefault}{\mddefault}{\updefault}{\color[rgb]{0,0,0}$\mcw(A)$}%
}}}}
\put(4321,-6541){\makebox(0,0)[lb]{\smash{{\SetFigFontNFSS{10}{12.0}{\familydefault}{\mddefault}{\updefault}{\color[rgb]{0,0,0}$\trop(A)$}%
}}}}
\put(5221,-2761){\makebox(0,0)[lb]{\smash{{\SetFigFontNFSS{10}{12.0}{\familydefault}{\mddefault}{\updefault}{\color[rgb]{0,0,0}$\mrw(A)$}%
}}}}
\put(4411,-3121){\makebox(0,0)[lb]{\smash{{\SetFigFontNFSS{10}{12.0}{\familydefault}{\mddefault}{\updefault}{\color[rgb]{0,0,0}$\termrk(A)$}%
}}}}
\put(5266,-3526){\makebox(0,0)[lb]{\smash{{\SetFigFontNFSS{10}{12.0}{\familydefault}{\mddefault}{\updefault}{\color[rgb]{0,0,0}$\ror(A)$}%
}}}}
\put(3781,-3526){\makebox(0,0)[lb]{\smash{{\SetFigFontNFSS{10}{12.0}{\familydefault}{\mddefault}{\updefault}{\color[rgb]{0,0,0}$\cor(A)$}%
}}}}
\put(4546,-4201){\makebox(0,0)[lb]{\smash{{\SetFigFontNFSS{10}{12.0}{\familydefault}{\mddefault}{\updefault}{\color[rgb]{0,0,0}$\facrk(A)$}%
}}}}
\end{picture}%
\caption{Comparison between ranks on $\rmax$}\label{fig1}
\end{figure}



We next show that the inequalities in Theorem~\ref{TEq} can be strict.
We already saw in Example~\ref{Xi} a matrix $X$ such that
$\facrk(X)=3=\cor(X) < \ror(X)=4$, which shows
that the two non central inequalities
at the fourth level (from the bottom) of Figure~\ref{fig1} may be strict
and that $\cor(A)$ and $\ror(A)$ may differ.
We note that the matrix $A=[\ooi \ooi]^t [\ooi \ooi]\in {\mathcal M}_{2}(\R)$ 
has term rank $2$ whereas 
$\facrk(A)=1=\ror(A)=\cor(A)=\mrw(A)=\mcw(A)$, showing that 
the central inequality at the same level may be strict
and that $\termrk(A)$ may differ from the other ranks under consideration.
Moreover, the matrix of Example~\ref{ex:mr_w} is such that $\ror(A)=\cor(A)=\mcw(A)<\mrw(A)$ showing that the inequalities at the fifth level of Figure~\ref{fig1} may be strict and that we may have $\mcw(A)\neq \mrw(A)$. 

To show that the other inequalities can be strict, we need some more
sophisticated examples.
%
\begin{ex} \label{ExfneK}
As in~\cite[Example 3.5]{DSS}, we consider the following matrix
$$
D_n=\left[ 
\begin{array}{rrrrr}
-1&\ooi&\ooi &\ldots &\ooi\\
\ooi& -1 & \ooi & \ldots & \ooi\\
\ooi& \ooi&-1 & \ldots & \ooi\\
\vdots & \vdots &\vdots &\ddots &\vdots \\
\ooi& \ooi & \ooi & \ldots & -1 \end{array}\right]\in\M(\R).
$$
\begin{enumerate}
\item It follows from~\cite[Proposition 2.2]{DSS}, that $\facrk(D_3)= 3$, $\facrk(D_4)=\facrk(D_5)=\facrk(D_6)=4$, $\facrk(D_7)=5$. 
\item It is easy to see that for $n\ge 3$, $\rk_{\det}(D_n)=3$. Indeed, any 4 rows (or columns) are Gondran-Minoux linearly dependent, and the positive and negative determinants of any principal $3 \times 3$-submatrix are different.
\item We have $\rk_{{\det}}(D_3)=3> 2 =\trop(D_3)$, showing that the
inequality at the first level (from the bottom) of Figure~\ref{fig1}
can be strict.
\item Since for $n\ge 4$ the sum of any two rows or columns of $D_n$ is equal
to $\unit$ and $\rk_{\det}(D_n)=3$, it follows that $\mrgm(D_n) =\mcgm(D_n) =3< \facrk(D_n)$,
showing that the inequalities at the third level of Figure~\ref{fig1}
can be strict. 
\end{enumerate}
\end{ex}
The following result shows that the maximal row and column Gondran-Minoux ranks may differ, and that they may also differ from the determinantal rank.
\begin{prop} \label{mrmcGM}
The matrix
$$
F=\left[ \begin{array}{ccccccc}
\ooo  & \ooi  & \ooi & \ooi & \ooi & \ooo & \ooo \\
\ooi  & \ooo  & \ooo & \ooo & \ooi & \ooo & \ooo \\
\ooi  & \ooo  & \ooi & \ooi & \ooo & \ooi & \ooo \\
\ooo  & \ooi  & \ooo & \ooo & \ooo & \ooi & \ooo \\
\ooi  & \ooi  & \ooo & \ooi & \ooo & \ooo & \ooi \\
\ooo  & \ooo  & \ooi & \ooo & \ooo & \ooo & \ooi 
\end{array}
\right]\in \MM_{6,7}(\rmax)
$$
is such that $\mrgm(F)=6 > \mcgm(F)=\rkdet(F)=5$.
\end{prop}
%
Before proving this proposition, let us explain the
idea leading to this example, which originates from~\cite{gaubert92a}. Consider
the matrix $F'$ over the symmetrized Boolean semiring 
$\BB$ (Definition~\ref{def-symb}),
\[
F'=\left[\begin{array}{cccc}
\ominus \unit &  \unit & \unit &\unit  \\
 \unit &  \ominus \unit & \unit &\unit  \\
\unit &   \unit & \ominus \unit &\unit  
\end{array}
\right]
\]
We make the following observations.

{\em Claim~1. Every maximal minor of $F'$ is balanced.}\/ Indeed,
any pair of columns of $F'$ contains a minor of order $2$ which is
equal to $\unit^\circ$, 
and so, when expanding any minor of order $3$ with respect to any column,
at least one of the terms in the expansion must be equal to $\unit^\circ$.

{\em Claim~2}. {\em There is no signed row vector $y\in (\smax)^3$ such that $y\neq\zero$ and $yF'\balance \zero$}.\/ Indeed, the columns of $F'$ contain all the vectors
$\alpha\in\{\unit,\ominus\unit\}^3$, up to a change of sign. Let us now
take $\alpha$ to be a sign vector of $y$, which is a vector
in $\{\unit,\ominus\unit\}^3$ such that $\alpha_i$ and
$y_i$ have the same sign (if $y_i=\zero$, the sign of $\alpha_i$ can be
chosen arbitrarily). Then, $y\alpha=\bigoplus_iy_i\alpha_i$ is invertible,
since the latter sum comprises only sign positive terms that cannot identically vanish because $y\neq\zero$. However, from $yF'\balance \zero$, 
we deduce that $y\alpha\balance\zero$, which is nonsense.

These two claims indicate that $F'$ is, loosely
speaking, of ``Gondran-Minoux maximal row rank'' $3$ but
of ``determinantal rank'' $2$, should we define these
notions in terms of balances. To obtain the desired counter example
for matrices in $\rmax$, 
it remains to ``double the structure'', which we do
by substituting $\unit$ and $\ominus\unit$ with the vectors
$[\unit,\zero]^t$ and $[\zero,\unit]^t$, respectively.
In this way, we arrive at the $6\times 4$ left submatrix of $F$. The remaining
$3$ columns are chosen precisely to encode the doubling of structure.

\begin{proof}[Proof of Proposition~\ref{mrmcGM}] \ \ 
1.~We first show that $\mrgm(F)=6$. Assume by contradiction that $\mrgm(F)<6$.
Then, we can find a signed non-zero row vector $x$ with entries in $\smax$
such that 
\[
xF \balance \ooo .
\]
Considering the last three columns of this vector relation,
we get
$$ x_1\oplus x_2\balance \ooo , \qquad x_3\oplus x_4\balance \ooo ,
\qquad x_5\oplus x_6\balance \ooo \enspace .$$ 
Since the entries of $x$ are signed, we deduce from Property~\ref{LRRR} 
that
\e
\label{eq:ls_x_123}
x_2=\ominus \, x_1 , \quad x_4=\ominus \, x_3, \quad x_6=\ominus \, x_5 \enspace.
\ee
Observe that $y:=[x_1,x_3,x_5]$ is signed and that it must be non-zero
due to the latter relations.  Substituting $x_2,x_4,x_6$ according
to~\eqref{eq:ls_x_123} in $xF\balance \zero$, and looking
only for the first four columns in the latter vector relation,
we arrive at $yF'\balance \zero$, where $F'$ is the matrix defined above.
Now, Claim~2 gives a contradiction, showing that the rows of $F$ are linearly independent in the Gondran-Minoux sense.


\medskip 2.~A straightforward computation shows that the seven maximal (i.e.\ $6\times 6$) minors of $F$ are balanced. Indeed, using the symmetry between
the three first columns, and the symmetry between the three last ones, it suffices to check that the three maximal minors obtained by suppressing the columns
$1$, $4$, or $7$ are unbalanced.
By the Gondran-Minoux theorem (see Corollary~\ref{TGM}), it follows that every family of $6$ columns of $A$ is linearly dependent in the Gondran-Minoux sense, and so $\mcgm(A)<6$. 
3.~Finally, a computation shows that the $\{2,3,4,5,6\}\times \{3,4,5,6,7\}$ submatrix of $F$ has an unbalanced determinant. Therefore, $\rkdet(F)\geq 5$. By Lemma~\ref{LO1new} applied to $F^t$, we get  $\rkdet(F)\leq \mcgm(F)$, and so $\rkdet(F)=\mcgm(F)=5$.
\end{proof}

%
%


\begin{ex}
As a corollary of the previous proposition, we get an example of a matrix $A$ with $\rkdet(A)<\min\{\mrgm(A),\mcgm(A)\}$. Indeed, let us consider the following block matrix:
$$
G:=\left[\begin{matrix}F &\ooo\\\ooo & F^t\end{matrix}\right]
\in \MM_{13,13}(\rmax) \enspace .
$$
Using the structure of $G$ and the previously established
properties of the matrix $F$, it can be checked that $\mrgm(G)=\mcgm(G)=11>10=\rkdet(G)$.
\end{ex}

\begin{problem} Find the minimal numbers $m$ and $n$ such that there exists an $m\times n$-matrix with different row and column Gondran-Minoux ranks.
\end{problem}

Recall that a family $\mathcal{I}$ of finite sets satisfies the
\NEW{augmentation axiom} of matroids if whenever $U,V\in \mathcal{I}$,  
if $V$ has more elements than $U$, we can find $v\in V$ in such a way
that $U\cup\{v\}\in \mathcal{I}$. 
The example of Proposition~\ref{mrmcGM} leads to the following
negative result.
\begin{corollary}
The set of finite subsets of $\rmax^n$ that are independent in the Gondran-Minoux sense does not satisfy the augmentation axiom of matroids.
\end{corollary}
\begin{proof}
Take $U$ to be the set of the $6$ rows of the matrix $F$ in Proposition~\ref{mrmcGM}, which was shown to be linearly independent (in the Gondran-Minoux sense),
and let $V$ be the set of the $7$ basis vectors $e_i$ of $\rmax^7$
($e_i$ has a coefficient $\unit$ in position $i$ and $\zero$ elsewhere).
The latter set is trivially linearly independent.
If the augmentation axiom held, we could add to $U$ one of the basis vectors
in order to get a $7\times 7$ matrix the rows of which are linearly
independent. 
By Theorem~\ref{th-gmnew}, the determinant of this matrix would be unbalanced, and by expanding it with respect to the last row, we would get
a $6\times 6$ maximal submatrix of $F$ with an unbalanced determinant, 
contradicting the fact that $\rk_{\det}(F)=5$.
\end{proof}
A simple example, given in~\cite{AGG}, shows that the set
of finite subsets of $\rmax^n$ that are tropically linearly independent
does not yield a matroid structure, either.


%
%


Finally, the following corollary points out a situation where the main rank functions
coincide.
\begin{corollary}
Let $A\in \Mnm(\R)$ be such that $\mrgm(A)=2$. Then $\trop(A)=\rkdet(A)=\mrgm(A)=\mcgm(A)=f(A)=\ror(A)=2$.
\end{corollary}
\d
Since $\trop(A)\le \mrgm(A)=2$, it follows that $\trop(A) $ is either 1 or 2
(excluding the trivial case where $A$ is the zero matrix). But if $\trop(A)=1$,
all the rows of $A$ would be proportional, contradicting $\mrgm(A)=2$. 
Hence, $\trop(A)=2$. 
Since every 3 rows of $A$ are Gondran-Minoux linearly dependent, one of these rows must be a linear combination of the others.
Therefore, $\ror(A)\leq 2$. Then, the result follows from Theorem~\ref{TEq}.
\dd

\begin{re}
The first part of Theorem~\ref{thbut} shows that if $A\in \MM_n(\rmax)$, 
it can be checked whether $\trop(A)=n$ in polynomial time. 
In~\cite{AGG}, we show that when $A\in \Mnm(\rmax)$, checking whether 
the tropical rank of $A$ is full, i.e., whether $\trop(A)=\min(m,n)$,
reduces to solving a mean payoff game. Thus, this problem belongs to
$\text{\bf NP} \cap \, \text{\bf co-NP}$, and is therefore
likely to be easy. 
This should be
opposed to a result of Kim and Roush~\cite{kim},
showing that the more general problem of computing $\trop(A)$
is $\text{\bf NP}$-hard.\typeout{TO BE CHECKED}
\end{re}

\section{Arithmetic behavior of rank functions}

In this section, we establish max-algebraic analogues of classical inequalities concerning the rank of the sum, product, or union of two matrices.

\begin{theorem}[\bf Rank-sum inequalities] \label{TSum}
For all matrices $A,B\in \Mnm(\R)$, the following inequalities hold:
\begin{enumerate}
\item 
$\facrk(A\oplus B) \le \facrk(A)+\facrk(B)$; \label{Sumf}
\item\label{rkdetsum} $\rk_{\det}(A\oplus B) \le \rk_{\det}(A)+ \rk_{\det} (B)$;
\item\label{troprksum} $ \trop(A\oplus B) \le \trop(A)+ \trop (B)$.
\end{enumerate}
\end{theorem}
\d
1.~The first inequality follows from~\cite[Proposition 4.2]{BG}.


2.~Let $A=(a_{ij})$, $B=(b_{ij})$. We denote $\rk_{\det}(A)=r_1$, $\rk_{\det}(B)=r_2$. Assume by contradiction that there is a minor of size $k:=r_1+r_2+1$ in the matrix $A\oplus B$ with different positive and negative determinants. 
From the monotonicity of $\rk_{\det}$ (Remark~\ref{monrank}),
 we may assume without loss of generality that $k=m=n$. Then, we can find
a permutation matrix $P$ and invertible diagonal matrices $D,D'$,
all with entries in $\rmax$, such that the matrix $C:=PD(A\oplus B)D'$ 
has the following properties: $C_{ij}\leq \unit$ and $C_{ii}=\unit$ for
all $i,j$. Indeed, such a transformation is obtained when
applying the Hungarian algorithm to solve the optimal
assignment problem for the matrix $A\oplus B$ (the scaling matrices
$D,D'$ coincide, up to a permutation of coordinates, with the optimal
variables of the dual linear problem). 
We shall assume without loss of generality that $C=A\oplus B$. 
In particular, $a_{ij}\le \ooi $, $b_{ij}\le \ooi$.
Since all diagonal entries of $C$ are equal to $\unit$, and all entries are
less or equal to $\unit$, we get that $|C|^+=\unit$, and $|C|^-\leq \unit$.
Moreover, by assumption $|C|^+$ and $|C|^-$ must be different,
hence $|C|^-<\unit$.

We denote $I_1=\{i\vert a_{ii}=\ooi\}$ and $I_2=\{i\vert b_{ii}=\ooi\}$, $s_1=\card{I_1}$, $s_2=\card{I_2}$. Note that $I_1\cup I_2=\{1,\ldots,n\}$, since all diagonal elements of $C=A\oplus B$ are equal to $\unit$, hence $s_1+s_2\ge n$.
This implies that either $s_1\ge r_1+1$ or $s_2\ge r_2+1$. Assume without loss of generality that $s_1\ge r_1+1$ and $I_1=\{1,\ldots, s_1\}$. 

Let $\hat A =A[I_1,I_1]$ and $\hat C=C[I_1,I_1]$ be the principal submatrices of $A$ and $C$ respectively, with rows and columns in $I_1$. Then, $|\hat A|^+=\unit$ and $|\hat A|^-\leq |\hat C|^-\leq \unit$, since all diagonal entries of $\hat A$ are equal to $\unit$, and all entries of $\hat C$ are less or equal to $\unit$.
Let us show that $|\hat C|^-< \unit$. 
Indeed, otherwise if $|\hat C|^-=\unit$, there exists an odd permutation $\sigma$ of $I_1$, such that 
$$C_{1\sigma(1)} \cdots  C_{s_1\sigma(s_1)}=\ooi\enspace.$$
Let $\tau$ be   the permutation of $\{1,\ldots, n\}$ such that $\tau(i)=\sigma(i)$ for $i=1,\ldots, s_1$ and $\tau(i)=i$ for $i=s_1+1,\ldots, n$. 
Since all diagonal entries of $C$ are equal to $\unit$, we get that 
\[ C_{1\tau(1)} \cdots  C_{n\tau(n)}=\ooi\enspace,\]
and since the permutation $\tau$ is odd, we deduce that 
$|C|^-=\unit$, a contradiction.
Hence $|\hat C|^-< \unit$, and since $|\hat A|^-\leq |\hat C|^-$ and
$|\hat A|^+=\unit$, we get that 
$|\hat A|^+\ne |\hat A|^-$, i.e., $\rk_{\det} A\ge s_1\ge r_1+1>r_1$. This contradiction concludes the proof.

3.~The proof of the third inequality is similar to the previous one, with the unique difference that we consider all the permutations of $\{1,\ldots, s_1\}$
and not only the odd ones.
\dd

\begin{re}
It is shown in~\cite[Proposition 7.2]{BG} that $1\le \ror(A\oplus B)$ and for any $r_1,r_2$ there are matrices of row ranks $r_1, r_2$ correspondingly such that their sum has row rank 1. Example~\ref{ExrSum} below shows that over the max-plus semiring, the row rank of the sum of two matrices may be also greater than the sum of their row ranks, so there is no reasonable upper bound for the row rank of a sum of matrices.
\end{re}

\begin{ex} \label{ExrSum}
Let us consider the following two matrices
$$
A=\left[\begin{array}{ccc} \ooo  & \ooi & \ooo  \\  \ooo  & \ooo  & \ooi \\ \ooo  & \ooo  & \ooi \\  \ooo  & \ooi & \ooo  \end{array}\right],\quad B= \left[\begin{array}{ccc} \ooo  & \ooo  & \ooo  \\  \ooo  & \ooo  & \ooo  \\ \ooi & \ooo  & \ooo  \\ \ooi & \ooo  & \ooo   \end{array}\right].
$$
Then it is straightforward to see that $\ror(A)=2$, $\ror(B)=1$, however $A\oplus B$ is the matrix $X$ from Example~\ref{Xi}. Thus $\ror(A\oplus B)=4>3=\ror(A)+\ror(B)$.
\end{ex}


\begin{theorem}[\bf Rank-product inequalities]
For all matrices $A\in \Mnm(\R)$, $B\in {\cal M}_{n\,k}(\R)$, the following inequalities hold:
\begin{enumerate}
\item $ \facrk(A  B) \le \min\{\facrk(A),\facrk(B)\}$;
\item $\rk_{\det}(A B) \le \min\{\rk_{\det}(A), \rk_{\det} (B)\}$;
\item $\trop(A B) \le \min\{\trop(A), \trop (B)\}$.
\end{enumerate}
\end{theorem}
\d
1.~The first inequality follows from~\cite[Proposition 4]{BG}.

2.~Recall that $C[I\vert J]$ denotes the $I\times J$ submatrix of a matrix $C$.
When both $I,J$ have  $s$ elements,
the strong form of the transfer principle (Theorem~\ref{transfer-s-2}), applied in the semiring $\smax$ to the
Binet-Cauchy formula (see Example~\ref{ex-binet-chauchy}) gives:
%
\begin{align*}
|(AB)[I\vert J]|
\succeq^\circ
\bigoplus\limits_{K\in Q_{s,n}} (|A[I\vert K]|\odot  |B[K\vert J]|)\enspace .
\end{align*}
By convention,
the sum is zero if $s>n$. Let $r:=\min(\rk_{\det}(A),\rk_{\det}(B))$.
If $s>r$, all the terms at the right hand side of the latter sum
are balanced. It follows that $|(AB)[I\vert J]|$ is balanced,
showing that $\rk_{\det}(AB)\leq r$.

3.~The third inequality is proved by replacing the semiring $\smax$ by the semiring $\Ti$ in the previous argument (recall that a square matrix with entries in $\rmax$ is tropically singular if and only if its determinant, when interpreted in $\Ti$, is balanced).

\dd

\begin{re}
Note that it may happen that $\ror(A B)>\ror(B)$ for some matrices $A$ and $B$, see Example~\ref{Exr(AB)}. 
\end{re}

\begin{ex} \label{Exr(AB)}
Let $$A=\left[ \begin{array}{cccc}
\ooi& \ooo&\ooo&  \ooo\\
\ooo&\ooi& \ooo&\ooo\\
\ooo& \ooi & \ooi& \ooo\\
\ooi&\ooo& \ooi &\ooo\end{array}\right]
, \quad B= \left[ \begin{array}{ccc}
 \ooo&\ooi&\ooo\\
\ooo&\ooo&\ooi\\
\ooi&\ooo&\ooo\\
\ooi &\ooo&\ooo\end{array}\right].
$$
Then $$A B= \left[ \begin{array}{ccc}
 \ooo&\ooi&\ooo\\
\ooo&\ooo&\ooi\\
\ooi&\ooo&\ooi\\
\ooi &\ooi&\ooo\end{array}\right].
$$
By using Example~\ref{Xi} we see that $\ror(A)=4$. It is straightforward to check that $\ror(B)=3$, however, $\ror(A B)=4$, cf. Example~\ref{Xi}.
\end{ex}
%


\begin{theorem}[\bf Ranks of matrix union]
For all $A\in \Mnm(\R)$ and $B\in {\cal M}_{m\,u}(\R)$ the following inequalities for the matrix union, denoted by $(A\vert B)$, hold: 
\begin{enumerate}
\item $ \max\{\ror(A),\ror(B)\}\le \ror(A\vert B) $; 
\item $\cor(A\vert B) \le \cor(A)+\cor(B)$ but it can be less than $\min\{\cor(A),\cor(B)\}$;
\item $\max\{\facrk(A),\facrk(B)\} \le \facrk(A\vert B) \le \facrk(A)+\facrk(B)$;
\item $\max\{\trop(A),\trop(B)\} \le \trop(A\vert B) \le \trop(A)+\trop(B)$;
\item $\max\{\rk_{\det}(A),\rk_{\det}(B)\} \le \rk_{\det}(A\vert B) \le \rk_{\det}(A)+\rk_{\det}(B)$.
\end{enumerate}
\end{theorem}
\begin{proof}
1.~The lower bound of $\ror(A\vert B)$ follows easily from the definition.
(We note that Example~\ref{ExrSum} shows that there is no reasonable
upper bound of this quantity.)


2.~The upper bound follows directly from the definitions.  Also for the matrices
$$
A=\left[\begin{array}{cccc} \ooi & \ooo & \ooi & \ooo \\
 \ooo & \ooi & \ooo & \ooi \\  \ooo & \ooo & \ooi & \ooi \end{array}\right], \quad B=\left[\begin{array}{cccc}  \ooo  & \ooo & \ooi & \ooi \\
 \ooo & \ooi & \ooo & \ooi \\ \ooi& \ooo & \ooi& \ooo  \end{array}\right]
$$
one has by Example~\ref{Xi} that $\cor(A)=4,\cor(B)=4$, but $\cor(A\vert B)=3<4$.

3.~Follows from~\cite[Lemma 3.17]{Psh}.

4.~The lower bound is evident. Observe that $(A\vert B)=[A,\zero]\oplus [\zero, B]$ where $\zero$ denotes the zero matrix of an arbitrary dimension. Then,
the upper bound follows from Theorem~\ref{TSum}, Assertion~\ref{troprksum}.

5.~Similarly, the lower bound is evident, whereas the upper bound follows from Theorem~\ref{TSum}, Assertion~\ref{rkdetsum}.
\end{proof}

\section*{Acknowledgments}


This paper was written when the third author was visiting the Maxplus team at INRIA, Paris - Rocquencourt, and INRIA, Saclay - \^Ile-de-France.  He would like to thank the colleagues from the both institutions for their warm hospitality. 

The authors thank the referee for his comments which led to some
improvements of the paper.






\def\cprime{$'$}

\end{document}